\documentclass[aop,preprint]{imsart}

\RequirePackage[OT1]{fontenc}
\RequirePackage{amsthm,amsmath,amsfonts}
  \usepackage{epsfig} 
\usepackage{graphicx}
\usepackage{float}

\RequirePackage[numbers]{natbib}
\RequirePackage[colorlinks,citecolor=blue,urlcolor=blue]{hyperref}

\startlocaldefs
\numberwithin{equation}{section}
\theoremstyle{plain}
\newtheorem{theorem}{Theorem}[section]
\newtheorem{lemma}[theorem]{Lemma}
\newtheorem{proposition}[theorem]{Proposition}

\theoremstyle{definition}
\newtheorem{definition}[theorem]{Definition}
\newtheorem{remark}[theorem]{Remark}

\endlocaldefs
\allowdisplaybreaks

\begin{document}

\begin{frontmatter}

\title 
{Dynamical system for animal coat pattern model}
\runtitle{T$\hat{\rm o}$n Vi$\hat{\d{e}}$t T\d{a}}

\begin{aug}

\author{T\^{o}n Vi$\hat{\d{e}}$t  T\d{a}
\ead[label=e1]{tavietton AT agr.kyushu-u.ac.jp}}

\address{Center for Promotion of International Education and Research 
\\
Faculty of Agriculture, Kyushu University
\\
6-10-1 Hakozaki, Higashi-ku, Fukuoka 812-8581, Japan\\
\printead{e1}}

\end{aug}

\begin{abstract}
We construct a dynamical system for a reaction diffusion system due to Murray, which relies on the use of the Thomas system nonlinearities and describes the formation of  animal coat patterns. 
First, we prove existence and uniqueness of global positive {\it strong} solutions  to the system by using semigroup methods. 
Second, we show that the solutions are continuously dependent on initial values. 
Third,  we show that the dynamical system enjoys  exponential attractors whose fractal dimensions can be estimated. 
Finally, we give a numerical example. 
\end{abstract}


\end{frontmatter}

\section {Introduction}

We consider a model of animal   coat patterns given by Murray (\cite{Murray81,Murray2003}):
\begin{equation} \label{Ton1}
\begin{cases}
\begin{aligned}
&\frac{\partial u }{\partial t}=\Delta u +\gamma \Big(a-u-\frac{\rho uv}{1+u+ku^2}\Big)    & \hspace{0.5cm} \text { in } \Omega \times (0,\infty),\\
&\frac{\partial v }{\partial t}  =\alpha \Delta v + \gamma \Big\{\beta(b-v)-\frac{\rho uv}{1+u+ku^2}\Big\}  &\hspace{0.4cm} \text { in } \Omega \times (0,\infty)
\end{aligned}
\end{cases}
\end{equation}
in a bounded domain $\Omega \subset \mathbb R^d \,(d=1,2,3\dots)$.
Here, $u(t)$ and $ v(t)$ denote the concentration of the activator and inhibitor at time $t$, respectively. These concentrations are supplied at constant rates $\gamma a$ and $\gamma\beta b$, and are degrade linearly proportional to themselves. Furthermore, both are used up in the reaction at a rate $f(u,v)=\frac{\gamma \rho uv}{1+u+ku^2}.$ The form of $f(u,v)$ exhibits substrate inhibition, where $k$ is a measure of the severity of the inhibition. The constant $\alpha>1$ is the ratio of diffusion coefficient. The constant $\gamma$ is a measure of the domain size, which can have any of the following interpretations:
\begin{itemize}
\item 
[(i)] $\gamma^{1/2}$ is proportional to the linear size of the spatial domain in one-dimension.
\item [(ii)] $\gamma$ represents the relative strength of the reaction terms. This means, for example, that an increase in $\gamma$ may represent an increase in activity of some rate-limiting step in the reaction sequence.
\item [(iii)] An increase in $\gamma$ can also be thought of as equivalent to a decrease in the diffusion coefficient ratio $\alpha$.
\end{itemize}

The system \eqref{Ton1} is coupled with a Neumann boundary condition:
\begin{equation} \label{Ton2}
\begin{cases}
\begin{aligned}
&\frac{\partial u}{\partial \mathbf {n}}=\frac{\partial v}{\partial \mathbf {n}}=0 
   &\hspace{1cm}  \text {   on } \partial \Omega \times (0,\infty),\\
& u(0,x)=u_0\geq 0, v(0,x)=v_0\geq 0   &\hspace{1cm} \text { in } \Omega,
\end{aligned}
\end{cases}
\end{equation}
where  $ \mathbf {n}$  denotes the  exterior normal to the boundary $ \partial \Omega. $

The system \eqref{Ton1} is a special form of the activator-inhibitor system, which is general given by a reaction diffusion system
\begin{equation} \label{Ton3}
\begin{cases}
\begin{aligned}
&\frac{\partial u }{\partial t}=\alpha_1 \Delta u +\gamma f(u,v)    & \hspace{1cm} \text { in } \Omega \times (0,\infty),\\
&\frac{\partial v }{\partial t}  =\alpha_2 \Delta v + \gamma g(u,v)  &\hspace{1cm} \text { in } \Omega \times (0,\infty).
\end{aligned}
\end{cases}
\end{equation}
By the use of this model, one can obtain many pattern formations in biological, physical, or chemical systems. Gierer-Meinhardt (\cite{Gierer-Meinhardt}) and Meinhardt (\cite{Meinhardt})  presented the functions
$$f=\nu_1-\nu_2 u+\frac{\nu_3 u^2}{v},$$
$$g=\nu_4 u^2-\nu_5v$$
for a model of biological pattern formation.  It is then called the Gierer-Meinhardt system. Global solutions to the system is then shown by Rothe (\cite{Rothe}). Existence of attractors and the Turing instability of the  system are given by Yagi (\cite{yagi}).

Masuda-Takahashi (\cite{Masuda-Takahashi}) then introduced  a generalized Gierer-Meinhardt system, i.e. the system \eqref{Ton3} with
$$f=\nu_1-\nu_2 u+\frac{\nu_3 u^p}{v^q},$$
$$g= \frac{\nu_4u^r}{v^s}-\nu_5v.$$
The authors then proved global existence of solutions in some special cases of coefficients.  Li-Chen-Qin (\cite{Li-Chen-Qin}) and then Jiang (\cite{Jiang}) showed global existence of solutions in some other cases. 

For the system \eqref{Ton1}, the nonlinearities came from the use of the Thomas system nonlinearities (\cite{Thomas}).  Murray reproduced animal coat patterns from this model, and gave a heuristic explanation for the formation of these patterns by investigating the linearization of the system \eqref{Ton1} at the unstable homogeneous stationary solution (\cite{Murray2003}). He concluded that these patterns depend only on the eigenfunction corresponding to the largest eigenvalue of the linearization. Sander-Wanner  then showed that the mechanism of these patterns is the same as that for spinodal decomposition in the Cahn-Hilliard equation (\cite{Cahn-Hilliard,Sander-Wanner}). This leads to a conclusion, which is in contrast to the Murray's explanation that
  the patterns in spaces of dimension less than three can be explained by linear behaviour corresponding to a whole range of largest eigenvalues.

The above interesting results on the system \eqref{Ton1} are based on an assumption: the system is well-posed. Not like to the Gierer-Meinhardt systems, global existence  as well as existence of an attractor to \eqref{Ton1}  have not been investigated. 

In this paper, we show that the system \eqref{Ton1} coupled with \eqref{Ton2} is well-posed. In other words, we prove existence and uniqueness of global positive {\it strong} solutions  to \eqref{Ton1}, and show that the solution's behavior changes continuously with initial conditions. For this, we use semigroup methods. We then construct a dynamical system for the model. Furthermore, by using the theory of dynamical systems (see Theorems \ref{THM1} and  \ref{THM2}) presented by Yagi (\cite{yagi}), we show that the dynamical system enjoys  exponential attractors whose fractal dimensions can be estimated. 

The paper is organized as follows. Section \ref{Preliminary} is preliminary. Some basic concepts such as sectorial operators, analytical semigroups, dynamical systems, and attractors are reviewed. 
In Section \ref{Abstract formulation}, we formulate the system \eqref{Ton1} into an abstract form, and recall the definition of mild and strong solutions to the form. 
A sufficient condition for existence of strong solutions to the abstract equation is presented.
In Section \ref{local solutions}, we construct local strong solutions, and prove the nonnegativity of these solutions. Section \ref{global solutions} shows that  the local strong solutions constructed in Section \ref{local solutions} are global by using a {\it priori} estimate for solutions.  Section \ref{initial data} provides the regular dependence of solutions on initial data. This helps us construct a continuous dynamical system in Section \ref{Dynamical system}.  Existence of exponential attractors is also shown in that section. The paper ends with a numerical example in Section \ref{example}.

\section{Preliminary}  \label{Preliminary}
Let $E$ be a Banach space with norm $\|\cdot\|$.  Let us review concepts of analytical semigroups generated by sectorial operators in $E$, and of dynamical systems on $E$. For more details, see, e.g.,  \cite{yagi}.

Throughout this paper, Banach and Hilbert spaces are always defined over the complex field $\mathbb C$.

\subsection{Sectorial operators and analytical semigroups}

A densely defined, closed linear operator $A$ in  $E$ is said to be sectorial if it satisfies the condition:
\begin{itemize}
  \item [(\rm{H})] The spectrum $\sigma(A)$ of $A$ is  contained in an open sectorial domain $\Sigma_{\varpi}$: 
\begin{equation*} \label{H1} 
\sigma(A) \subset  \Sigma_{\varpi}=\{\lambda \in \mathbb C: |\arg \lambda|<\varpi\}, \quad \quad 0<\varpi<\frac{\pi}{2}.
       \end{equation*}
  The resolvent of $A$ satisfies the estimate  
\begin{equation*} \label{H2}
          \|(\lambda-A)^{-1}\| \leq \frac{M_{\varpi}}{|\lambda|}, \quad\quad\quad \quad   \lambda \notin \Sigma_{\varpi}
     \end{equation*}
     with some  constant $M_{\varpi}>0$ depending only on the angle $\varpi$.
     \end{itemize}

Let $A$ be a  sectorial operator. The fractional powers $A^\theta, -\infty<\theta<\infty,$ are then defined as follows. For each complex number $z$ such that ${\rm Re}\, z>0$, $A^{-z}$ is defined by using the Dunford integral in $\mathcal L(E)$:
$$A^{-z}=\frac{1}{2\pi i} \int_\gamma \lambda^{-z} (\lambda-A)^{-1} d\lambda.$$
Here, $\gamma=\gamma_-\cup \gamma_0 \cup \gamma_+$ is an integral contour surrounding the spectrum $\sigma(A)$ counterclockwise in the domain $\mathbb C\setminus (-\infty,0] \cap \mathbb C\setminus\sigma(A)$ of the complex plane, where 
\begin{align*}
\gamma_\pm \colon \lambda=\rho e^{\pm i \varpi}, \hspace{2cm} \frac{\|A^{-1}\|^{-1}}{2} \leq \rho<\infty,
\end{align*}
and 
\begin{align*}
\gamma_0 \colon \lambda=\frac{\|A^{-1}\|^{-1}}{2} e^{ i \varphi}, \hspace{1cm}  -\varpi \leq \varphi\leq \varpi.
\end{align*}
It is known that $A^{-z}$ is one to one for ${\rm Re}\, z>0$. The following definition is thus meaningful:
$$A^z=(A^{-z})^{-1} \hspace{1cm} \text{ for } {\rm Re}\, z>0.$$
In addition, it is natural to define $A^0=I$, the identity mapping on $E$.  
In this way, for every real number $-\infty<\theta<\infty,$ $A^\theta$ has been defined. 

The following lemma shows useful estimates for fractional powers and the semigroup generated by a sectorial operator.
\begin{lemma}   
Let {\rm (H)}   be satisfied. Then,
\begin{itemize}
\item  [\rm (i)] $(-A)$ generates an analytical  semigroup $\{e^{-tA}, t\geq 0\}.$
\item  [\rm (ii)] For $0\leq \theta  <\infty,$
\begin{equation} \label{Ton4}
\|A^\theta  e^{-tA} \| \leq \iota_\theta t^{-\theta}, \hspace{2cm} 0<t<\infty,   
\end{equation}
 where $ \iota_\theta=\sup_{0\leq  t<\infty} t^\theta \|A^\theta  e^{-tA} \|<\infty$. In particular, there exists $0<\nu<\infty$ such that
\begin{equation} \label{Ton5}
\| e^{-tA} \| \leq \iota_0 e^{-\nu t}\leq \iota_0, \hspace{2cm}  0\leq t<\infty.
\end{equation}
 \item  [\rm (iii)]  For  $0<\theta\leq 1,$ 
\begin{equation} \label{Ton6}
\|[ e^{-tA} -I]A^{-\theta}\|\leq \frac{\iota_{1-\theta}}{\theta} t^\theta, \hspace{2cm}  0\leq t<\infty.
\end{equation}
\end{itemize}
\end{lemma}
For the proof, see \cite{yagi}.

\subsection{Dynamical systems}
A family of nonlinear operators $S(t), 0\leq t<\infty$, from a subset $E_0$ of $E$  into $E_0$ is called a semigroup on $E_0$ if the family satisfies
\begin{itemize}
 \item [{ \rm (i)}] $S(0)=I$.
 \item  [{ \rm (ii)}] $S(t)S(s)=S(t+s), \hspace{1cm} 0\leq s, t<\infty$.
  \end{itemize}

In addition to these properties, if the operators $S(t)$ satisfy
\begin{itemize}
 \item [{ \rm (iii)}] $S(\cdot)(\cdot)$ is continuous from $[0,\infty)\times E_0$ to $E_0$
 \end{itemize}
then the family is called a continuous semigroup on $E_0$.

Let $S$ be a continuous semigroup on $E_0$. For each $x\in E_0$, the trace of $E_0$\,{-}\,valued continuous function $S(\cdot) x$ in $E_0$ is called a trajectory starting from $x$. A dynamical system is the set of trajectories starting from the points in $E_0$, and is denoted by  a triple $(S, E_0, E).$

A set $\mathcal A\subset E_0$ is called an absorbing set if $\mathcal A$ absorbs every bounded set of $E_0$, i.e. for any bounded set $B$ of $E_0$, there exists $0\leq t_B<\infty$  such that 
$$S(t) B \subset \mathcal A, \hspace{1cm} t_B\leq t<\infty.$$

A set $\mathcal A\subset E_0$ is called an attractor if it satisfies two conditions: 
 \begin{itemize}
 \item [{ \rm (i)}] $\mathcal A$ is an invariant set of the dynamical system, i.e. 
    $$S(t)\mathcal A \subset \mathcal A, \hspace{1cm} 0\leq t<\infty.$$
 \item  [{ \rm (ii)}] For some neighborhood $W$ of $\mathcal A$,
   $$\lim_{t\to \infty} h(S(t)W,\mathcal A)=0.$$ 
 Here, $h(\cdot,\cdot)$ is the Hausdorff pseudo-distance, i.e. for two subsets $B_1$ and $B_2$ of $E$, 
$$h(B_1, B_2)=\sup_{x\in B_1} \inf_{y\in B_2} \|x-y\|.$$
 \end{itemize}

An attractor $\mathcal A$ is called a global attractor of the dynamical system $(S, E_0, E)$ if it satisfies three conditions: 
 \begin{itemize}
 \item [{ \rm (i)}] $\mathcal A$ is a compact set of $E.$
 \item  [{ \rm (ii)}] $\mathcal A$ is a strictly  invariant set of the dynamical system, i.e. 
    $$S(t)\mathcal A = \mathcal A, \hspace{1cm} 0\leq t<\infty.$$
 \item  [{ \rm (ii)}] $\lim_{t\to \infty} h(S(t)B,\mathcal A)=0$ for every bounded set $B$ of $E_0$.
 \end{itemize}

Let us finally review the concept of exponential attractors.  A compact attractor  $ \mathcal A^*\subset E_0$  is called an exponential attractor if it satisfies the following conditions:
\begin{itemize}
  \item [{ \rm (i)}]  $\mathcal A^*$ contains a global attractor $\mathcal A$ of the dynamical system $(S, E_0, E).$
  \item [{ \rm (ii)}] $\mathcal A^*$ has finite fractal dimension. Here, the fractal dimension of $\mathcal A^*$ is defined by
$$d_F(\mathcal A^*)=-\limsup_{\epsilon \to 0} \frac{\log N(\epsilon)}{\log \epsilon},$$
where $N(\epsilon)$ is the minimal number of balls with radius $\epsilon$ which cover the set $\mathcal A^*.$
  \item [{ \rm (iii)}] There exists $0<\alpha<\infty$ such that for all bounded set $B\subset E_0$, 
$$h(S(t)B, \mathcal A^*) \leq C_B e^{-\alpha t},\hspace{1cm} 0\leq t<\infty$$
with some $0<C_B<\infty.$
  \end{itemize}

\begin{remark}
It is obvious that every dynamical system enjoys  at most one global attractor. However, an exponential attractor, if it exists, is not unique in general. Exponential attractors exist as a family. 
\end{remark}

 The following theorem shows a sufficient condition for existence of exponential attractors for a dynamical system.

\begin{theorem}   \label{THM1}
Let $(S,E_0,E)$ be a dynamical system in a Banach space $E$, where the phase space $E_0$ is a closed bounded set of $E$. Assume that there exist an operator $K$, two constants $0<t_0<\infty$ and $0\leq \delta<\frac{1}{2}$ satisfying the following conditions:
\begin{itemize}
  \item [{ \rm (i)}] $K\colon E_0\to E_1$, where $E_1$ is a Banach space which is compactly embedded in $E$. In addition, $K$ satisfies a Lipschitz condition in the sense
  $$\|K(x)-K(y)\|_{E_1} \leq L\|x-y\|, \hspace{1cm} x,y\in E_0$$
  with some $L>0.$  
  \item [{ \rm (ii)}]$ S(t_0)$ is a compact perturbation of contraction operator  in the sense
    $$\|S(t_0)(x)-S(t_0)(y)\| \leq \delta\|x-y\| +\|K(x)-K(y)\|, \hspace{1cm} x,y\in E_0.$$
  \item [{ \rm (iii)}] $S(\cdot)(\cdot) $ satisfies a Lipschitz condition on $[0,t_0] \times E_0$ in the sense
   $$\|S(t)(x)-S(s)(y)\| \leq L_1(|t-s| + \|x-y\|), \hspace{1cm} 0\leq t,s\leq t_0, x,y\in E_0$$
  with some $L_1>0.$
 \end{itemize} 
Then, for any $0<\theta<\frac{1-2\delta}{2L}$, there exists an exponential attractor $\mathcal A_\theta$ for $(S,E_0,E)$ such that
$$h(S(t)E_0,\mathcal A_\theta) \leq \frac{R}{2(\delta+\theta L)} e^{\frac{t\log [2(\delta+\theta L)] }{t_0}}, \hspace{ 1cm} 0<t <\infty.$$
Here, $R$ is the diameter of $E_0$.
\end{theorem}

For the proof, see  \cite{yagi}.

Let us now consider the case where $E$ is a Hilbert space and $E_0$ is a compact subset of $E$. We say that an operator $S_1$ from $E_0$ to $E_0$ has {\it squeezing property} if there exist a constant $0<\delta<\frac{1}{4}$ and an orthogonal operator $P$ of finite rank $N$ such that for $x, y\in E_0$, either 
$$\|S_1(x)-S_1(y)\| \leq \delta \|x-y\|$$
or
$$(I-P)(S_1(x)-S_1(y))  \leq \| P(S_1(x)-S_1(y))\|.$$

\begin{theorem} \label{THM2}
Let $(S,E_0,E)$ be a dynamical system in a Hilbert space $E$, where $E_0\subset E$ is  compact. Assume that for some fixed $0<t_0<\infty,$  the operator  $S(t_0)$  has the squeezing property with some $0<\delta<\frac{1}{4}$ and an orthogonal operator $P$ of finite rank $N.$ Assume further that $S(t)$ satisfies the Lipschitz condition {\rm (iii)} of Theorem \ref{THM1}.

Then, for any $0<\theta<1-2\delta$, there exists an exponential attractor $\mathcal A_\theta$ for $(S,E_0,E)$ such that
$$h(S(t)E_0,\mathcal A_\theta) \leq \frac{R}{2\delta+\theta } e^{\frac{t\log (2\delta+\theta) }{t_0}}, \hspace{ 1cm} 0<t <\infty.$$
In addition, the fractal dimension of $\mathcal A_\theta$ is finite and estimated by
$$d_F(A_\theta) \leq 1 + N \max\{  -\frac{  \log (\frac{3L_1}{\theta}+1)    }  {\log (2\delta+\theta)}, 1   \}.$$
Here, $R$ is the diameter of $E_0$, i.e. $R=\sup_{x, y\in E_0}\|x-y\|.$
\end{theorem}

For the proof, see  \cite{yagi}.

\section{Abstract formulation}  \label{Abstract formulation}
Let us formulate  the system \eqref{Ton1} coupled with the condition \eqref{Ton2} as an abstract equation.

Set the underlying space
\begin{equation*} 
(E,\|\cdot\|)=(L_2(\Omega ) \times L_2(\Omega ),\|\cdot\|_{L_2 \times L_2}), 
\end{equation*}
where the norm is defined by
$\Big\|\Big(
\begin{matrix} 
x_1\\
x_2
\end{matrix}\Big)\Big\|_{L_2 \times L_2}=\sqrt{\|x_1\|_{L_2}^2+ \|x_2\|_{L_2}^2},$  $\Big(
\begin{matrix} 
x_1\\
x_2
\end{matrix}\Big)\in E.$
And set the space of initial functions
\begin{equation}  \label{Ton7}    
V=\left \{\left(
\begin{matrix} 
u_0\\
v_0
\end{matrix}\right); 0\leq u_0, v_0 \in L_2(\Omega ) \right\}.
\end{equation}

Denote by $A$ a diagonal matrix operator 
diag$\{A_1,A_2\}$ of $E$, where $A_1$ and $A_2$ are realization of operators $-\Delta+\gamma$ and $-\alpha\Delta+\gamma \beta$ in 
$L_2(\Omega )$, respectively, under the homogeneous Neumann boundary condition \eqref{Ton2} on $\partial \Omega $.
According to \cite[Theorem 1.25]{yagi}, $A_i \, (i=1,2)$ are positive define self-adjoint operators of $L_2(\Omega )$ with domain
$$\mathcal D(A_i)=H^2_N(\Omega )=\{h\in H^2(\Omega); \frac{\partial h}{\partial \mathbf {n}}=0 \text{  on  } \partial \Omega\},$$
here $H^2(\Omega)$ denotes the space of all complex-valued functions whose partial derivatives in the distribution sense up to the second order belong to $L_2(\Omega ).$
Furthermore, on account of \cite[Theorem 16.7, 16.9]{yagi}, the domain of the fractional power $A^\theta$ is defined by 
\begin{equation}  \label{Ton8}
\mathcal D(A^\theta)=
\begin{cases}
H^{2\theta}(\Omega )\times H^{2\theta}(\Omega ), \hspace{1cm} \text{if } 0\leq \theta<\frac{3}{4},\\
H^{2\theta}_N(\Omega ) \times H^{2\theta}_N(\Omega ), \hspace{1cm} \text{if } \frac{3}{4}< \theta\leq 1.
\end{cases}
\end{equation}
(Here, 
$$H^{2\theta}_N(\Omega )=\{h\in H^{2\theta}(\Omega); \frac{\partial h}{\partial \mathbf {n}}=0 \text{  on  } \partial \Omega\},$$
and 
$H^{2\theta}(\Omega)$ is  the  space of all restrictions of functions in $H^{2\theta}(\mathbb R^n)$ 
to $\Omega$. In addition,  
$$H^{2\theta}(\mathbb R^n)=\{f\in \mathcal S(\mathbb R^n)'; \mathcal F^{-1} [(1+|x|^2)^\theta \mathcal Ff] \in L_2(\mathbb R^n)\},$$
where $\mathcal S(\mathbb R^n)'$ is the the set of tempered distributions in $\mathbb R^n$, $\mathcal F^{-1} $ and $\mathcal F$ are the Fourier transform and inverse Fourier transform on $\mathcal S(\mathbb R^n)'$, respectively.)

Let us define an operator $\bar F \colon  E\to E$ by
\begin{equation*}  
\bar F(X)=\left (
\begin{matrix} 
\gamma a-\frac{\gamma \rho u|v|}{1+|u|+ku^2}\\
\gamma \beta b-\frac{\gamma \rho |u|v}{1+|u|+ku^2}
\end{matrix}\right), \quad X=\left(
\begin{matrix} 
u\\
v
\end{matrix}\right) \in E.
\end{equation*}
Since 
$$\Big| \frac{u}{1+|u|+ku^2}\Big|\leq 1,$$ 
$\bar F$ certainly maps $E$ into $E$. Furthermore, the function $\bar F$ has the following properties.

\begin{lemma}   \label{THM3}
There exists $0<c_\infty<\infty$ such that
\begin{equation*}
\|\bar F(x)-\bar F(y)\| \leq c_\infty(1+\|x\|_{\infty}+\|y\|_{\infty})\|x-y\|, \hspace{1cm} x, y\in E,
\end{equation*}
and 
\begin{equation*}
\|\bar F(x)\| \leq c_\infty (1+\|x\|), \hspace{1cm} x\in E.
\end{equation*}
Here, $\|\cdot\|_{\infty}$ is defined by 
$$\|x\|_{\infty}=\max\{\|x_1\|_{L_\infty}, \|x_2\|_{L_\infty}\}, \quad 
 x=\left(
\begin{matrix} 
x_1\\
x_2
\end{matrix}\right) \in E.$$
\end{lemma}

\begin{proof}
Let $x=\left(
\begin{matrix} 
u_1\\
v_1
\end{matrix}\right)$
 and 
$y=\left(
\begin{matrix} 
u_2\\
v_2
\end{matrix}\right)$ in $E$. 
Then,
$$\bar F(x)-\bar F(y)=\gamma \rho \left(
\begin{matrix} 
\frac{u_2|v_2|}{1+|u_2|+ku_2^2}-\frac{ u_1|v_1|}{1+|u_1|+ku_1^2}\\
\frac{ |u_2|v_2}{1+|u_2|+ku_2^2}-\frac{ |u_1|v_1}{1+|u_1|+ku_1^2}
\end{matrix}\right).
$$

We consider the norm in $L_2$ of the first component in the latter parenthesis. We have
\begin{align*}
&\left\|\frac{ u_2|v_2|}{1+|u_2|+ku_2^2}-\frac{ u_1|v_1|}{1+|u_1|+ku_1^2}\right\|_{L_2} \\
\leq & \left\|\frac{ u_2|v_2|  - u_1|v_1|}  {(1+|u_2|+ku_2^2)(1+|u_1|+ku_1^2)}
\right\|_{L_2}    \\
&+\left\|\frac{ |u_1| u_2|v_2|-u_1|v_1||u_2|}  {(1+|u_2|+ku_2^2)(1+|u_1|+ku_1^2)}
\right\|_{L_2}  \\
&+\left\|\frac{ k(u_1^2u_2 |v_2|-u_1 |v_1| u_2^2)}  {(1+|u_2|+ku_2^2)(1+|u_1|+ku_1^2)}
\right\|_{L_2} \\
=& \bar F_{11}+\bar F_{12}+\bar F_{13}.
\end{align*}

The first term, $\bar F_{11}$, can be estimated easily:
\begin{align*}
\bar F_{11} \leq & \|u_2|v_2|  - u_1|v_1|\|_{L_2}  \\
= & \|(u_2-u_1) |v_2|  + u_1(|v_2|-|v_1|)\|_{L_2}  \\
\leq & \|u_2-u_1\|_{L_2} \|v_2\|_{L_\infty}  + \|u_1\|_{L_\infty}\|v_2-v_1\|_{L_2}  \\
\leq & \|x-y\| (\|x\|_{\infty}  + \|y\|_{\infty}).
\end{align*}

For the second term, $\bar F_{12}$,  we have
\begin{align*}
\bar F_{12} = & \left\|\frac{(u_2-u_1) |u_1| |v_2| + u_1 (|u_1| |v_2| - |v_1| |u_2|)}  {(1+|u_2|+ku_2^2)(1+|u_1|+ku_1^2)}
\right\|_{L_2}   \\
= & \left\|\frac{(u_2-u_1) |u_1| |v_2| + u_1 |v_2| (|u_1|-|u_2|) + u_1 |u_2|(|v_2|-|v_1|)}  {(1+|u_2|+ku_2^2)(1+|u_1|+ku_1^2)}
\right\|_{L_2}   \\
\leq & \left\|\frac{(u_2-u_1) |u_1| |v_2| }  {1+|u_1|+ku_1^2}
\right\|_{L_2}   
+\left\|\frac{ u_1 |v_2| (|u_1|-|u_2|)}  {1+|u_1|+ku_1^2}
\right\|_{L_2}  \\
&+\left\|\frac{u_1 |u_2|(|v_2|-|v_1|)}  {1+|u_1|+ku_1^2}
\right\|_{L_2}  \\
\leq & \|(u_2-u_1)  |v_2| \|_{L_2}   
+\| |v_2| (|u_1|-|u_2|)\|_{L_2}  
+\| |u_2|(|v_2|-|v_1|)\|_{L_2}  \\
\leq & 2\|u_2-u_1   \|_{L_2}   \|v_2\|_{L_\infty}
+\|u_2\|_{L_\infty} \|v_2-v_1\|_{L_2}  \\
\leq & 3\|x-y\| \|y\|_\infty.
\end{align*}

For the last term, $\bar F_{13}$,  we have
\begin{align*}
\bar F_{13} = & \left\|\frac{ k(u_1^2u_2 (|v_2|-|v_1|)
+ku_1|v_1| u_2(u_1-u_2)}  {(1+|u_2|+ku_2^2)(1+|u_1|+ku_1^2)}
\right\|_{L_2} \\
\leq  & \left\|\frac{ ku_1^2u_2 (|v_2|-|v_1|)}  {1+ku_1^2 |u_2|}
\right\|_{L_2}  
+\left\|\frac{ku_1 u_2|v_1|(u_1-u_2)}  {1+|u_1 u_2|}
\right\|_{L_2}  \\
\leq & \|v_2-v_1 \|_{L_2}   
+k\|v_1\|_{L_\infty} \|u_1-u_2\|_{L_2}  \\
\leq & (1+k\|x\|_\infty) \|x-y\|.
\end{align*}

The above estimates for the three terms give that
\begin{align*}
&\left\|\frac{ u_2|v_2|}{1+|u_2|+ku_2^2}-\frac{ u_1|v_1|}{1+|u_1|+ku_1^2}\right\|_{L_2} \\
&\leq \|x-y\| [1+ (k+1)\|x\|_\infty+4\|y\|_\infty )] \\
&\leq \max\{k+1,4\} (1+\|x\|_\infty+\|y\|_\infty)  \|x-y\|.
\end{align*}

Similarly, we have an estimate for the norm in $L_2$ of the second component:
\begin{align*}
&\left\|\frac{ |u_2|v_2}{1+|u_2|+ku_2^2}-\frac{ |u_1|v_1}{1+|u_1|+ku_1^2}
\right\|_{L_2} \\
&\leq \max\{k+1,4\} (1+\|x\|_\infty+\|y\|_\infty)  \|x-y\|.
\end{align*}
Therefore, we arrive at
\begin{equation}   \label{Ton9}   
\|\bar F(x)-\bar F(y)\| \leq \sqrt {2} \gamma \rho \max\{k+1,4\} (1+\|x\|_{\infty}+\|y\|_{\infty})\|x-y\|.
\end{equation}

In the meantime, we have
 \begin{align*}
\|\bar F(x)\|^2 =&\gamma^2 \left\|a-\frac{ \rho u_1 v_1}{1+|u_1|+ku_1^2}
\right\|_{L_2}^2 
+\gamma^2 \left\|\beta b-\frac{ \rho |u_1| v_1}{1+|u_1|+ku_1^2}
\right\|_{L_2}^2 \\
&\leq \gamma^2 (a \text {Vol}(\Omega) + \rho \|v_1\|_{L_2})^2+\gamma^2 (\beta b \text {Vol}(\Omega) + \rho \|v_1\|_{L_2})^2  \\
&\leq 2\gamma^2 [\max\{a, \gamma b\}  \text {Vol}(\Omega) + \rho \|x\|]^2,
\end{align*}
where $\text {Vol}(\Omega)=\int_\Omega dz.$ 
This means that
\begin{equation}   \label{Ton10}
\|\bar F(x)\|\leq \sqrt {2} \gamma [\max\{a, \gamma b\}  \text {Vol}(\Omega) + \rho \|x\|].
\end{equation}

Thus, the lemma has been proved due to \eqref{Ton9} and \eqref{Ton10}.
\end{proof}

Using $A$ and $\bar F$,  the system  \eqref{Ton1} coupled with  \eqref{Ton2} is formulated as an abstract equation of the form 
\begin{equation} \label{Ton11}
\begin{cases}
\frac{dX}{dt}+AX=\bar F(X), \hspace{1cm} 0< t<\infty,\\
X(0)=X_0  \in V
\end{cases}
\end{equation}
 in $E$, where 
$X_0=\left(
\begin{matrix} 
u_0\\
v_0
\end{matrix}\right).
$
At the end of Section  \ref{local solutions}, we conclude that a solution of \eqref{Ton11} is also a solution of \eqref{Ton1} coupled with  \eqref{Ton2}.

\begin{definition}  \label{THM4}
\begin{itemize}
  \item [(i)] 
A continuous function $X$ on $[0,\infty)$ is called a mild solution to \eqref{Ton11} if
$$X(t)= e^{-tA}  X_0+\int_0^t e^{-(t-s)A}\bar F(X(s))ds, \hspace{1cm} 0\leq t<\infty,$$
where $\{e^{-tA}, t\geq 0\}$ is the semigroup generated by operator $A$ in $E$.
\item  [(ii)] A continuous function $X$  on $[0,\infty)$ is called a strong solution to \eqref{Ton11} if 
    $$X\in \mathcal C^1((0,\infty);E), \quad X(t)\in \mathcal D(A), \hspace{1cm} 0<t<\infty,$$
     and $X$ satisfies \eqref{Ton11}.
\end{itemize}
\end{definition}

\begin{proposition}  \label{THM5}
\begin{itemize}
  \item [{\rm (i)}]
A strong solution of  \eqref{Ton11} is always a mild solution. 
\item [{\rm (ii)}] Let $X$ be a mild solution of  \eqref{Ton11}. Assume that the function $G$   defined by
   $$G(t)=\int_0^t (t-s)^{-1} \|\bar F(X(s))-\bar F(X(t))\| ds,$$
is defined and integrable on any closed interval of $(0, \infty)$.
Then, $X$ becomes a strong solution of \eqref{Ton11}.
\end{itemize}
\end{proposition}

\begin{proof}
The part {\rm (i)} is quite obvious. So we only prove the part {\rm (ii)}.

Let $A_n, n=1, 2, 3 \dots,$ be the Yosida approximation of the operator $A$ defined by 
\begin{equation*}
A_n=A(1+\frac{A}{n})^{-1}=n-n^2(n+A)^{-1}, \hspace{1cm} n=1,2,3\dots
\end{equation*}
It is well known that 
\begin{itemize}
  \item $A_n$ generates an analytical semigroup $\{e^{-tA_n}, t\geq 0\}.$ Furthermore, 
  \begin{equation}   \label{Ton12}
    \begin{cases}
     \lim_{n\to \infty} A_n x=Ax \hspace{1cm} \text{ in } E, x\in \mathcal D(A), \\
    \lim_{n\to \infty} e^{-tA_n}=e^{-tA} \hspace{1cm} \text{ in } \mathcal L(E).
    \end{cases}
    \end{equation}
  \item For $0\leq \theta  <\infty,$
\begin{equation} \label{Ton13}
\|A_n^\theta e^{-tA_n}\| \leq \iota_\theta t^{-\theta}, \hspace{2cm} 0<t<\infty,   
\end{equation}
 and
\begin{equation} \label{Ton14}
\|e^{-tA_n}\| \leq \iota_0 e^{-\nu t}\leq \iota_0, \hspace{2cm}  0\leq t<\infty,
\end{equation}
where $ \iota_\theta$ and $\nu$ are the constants in \eqref{Ton4} and \eqref{Ton5}.
  \end{itemize}

Consider a function 
\begin{equation*}
X_n(t)=e^{-tA_n} X_0 +\int_0^t e^{-(t-u)A_n} \bar F(X(u)) du, \hspace{1cm} 0\leq t<\infty.
\end{equation*}
Then, \eqref{Ton12} and \eqref{Ton14} give 
$$\lim_{n\to \infty} X_n(t) =X(t), \hspace{1cm} 0\leq t<\infty.$$
In addition, since $A_n$ is a bounded operator, we have
\begin{align*}
\frac{dX_n}{dt}=&-A_n e^{-tA_n} X_0  - \int_0^t A_n e^{-(t-u)A_n} \bar F(X(u)) du +\bar F(X(t)) \\
=&-A_n X_n(t)+\bar F(X(t)), \hspace{1cm} 0<t<\infty.
\end{align*}
Therefore, for any $\epsilon>0$,
\begin{equation}  \label{Ton15}
X_n(t)=X_n(\epsilon) +\int_\epsilon^t [\bar F(X(u))-A_nX_n(u)]du, \hspace{1cm} \epsilon \leq t<\infty.
\end{equation}

We now want to show the convergence of $A_nX_n(t)$ for $0<t<\infty$. We have
\begin{align}
A_nX_n(t)=& A_n e^{-tA_n} X_0 +\int_0^t A_n e^{-(t-u)A_n} \bar F(X(u)) du \notag\\
=&  A_n e^{-tA_n} X_0 +\int_0^t A_n e^{-(t-u)A_n} [\bar F(X(u))-\bar F(X(t))] du  \notag\\
&+\int_0^t A_n e^{-(t-u)A_n} du \bar F(X(t))     \notag \\
=&   A_n e^{-tA_n} X_0 +\int_0^t A_n e^{-(t-u)A_n} [\bar F(X(u))-\bar F(X(t))] du    \label{Ton16}\\
&+[I-e^{-tA_n}]\bar F(X(t)).  \notag
\end{align}
The latter integral is bounded due to  \eqref{Ton13}: 
\begin{align}
& \int_0^t \|A_n e^{-(t-u)A_n} [\bar F(X(u))-\bar F(X(t))]\| du    \notag\\
& \leq \int_0^t (t-u)^{-1}\|\bar F(X(u))-\bar F(X(t))\|du=G(t).   \label{Ton17}
\end{align}
Hence,  the Lebesgue dominate theorem provides that 
\begin{align*}
\lim_{n\to \infty} & \int_0^t A_n e^{-(t-u)A_n} [\bar F(X(u))-\bar F(X(t))] du\\
&=\int_0^t A e^{-(t-u)A} [\bar F(X(u))-\bar F(X(t))] du.
\end{align*}
Thus,
\begin{align*}
\lim_{n\to \infty} A_nX_n(t) =& A e^{-tA} X_0 +\int_0^t A e^{-(t-u)A} [\bar F(X(u))-\bar F(X(t))] du\\
&+[I-e^{-tA}]\bar F(X(t)).
\end{align*}

Using this, we have
\begin{align*}
X(t)=&\lim_{n\to \infty} A_n^{-1} A_nX_n(t) \\
=& A^{-1} [A e^{-tA} X_0 +\int_0^t A e^{-(t-u)A} [\bar F(X(u))-\bar F(X(t))] du\\
&+[I-e^{-tA}]\bar F(X(t)).
\end{align*}
This shows that
$$X(t) \in \mathcal D(A(t)),$$
 and
\begin{align*}
AX(t)=&A e^{-tA} X_0 +\int_0^t A e^{-(t-u)A} [\bar F(X(u))-\bar F(X(t))] du\\
&+[I-e^{-tA}]\bar F(X(t)), \hspace{1cm} 0<t<\infty.
\end{align*}

Let us next consider the convergence in \eqref{Ton15} when $n\to \infty$. Thanks to \eqref{Ton16} and \eqref{Ton17},
\begin{align*}
\|A_nX_n(t)\|\leq & t^{-1} \|X_0\| +\int_0^t (t-u)^{-1}\|\bar F(X(u))-\bar F(X(t))\|du \notag\\
&+(1+\iota_0) \|\bar F(X(t))\|  \notag  \\
=&t^{-1} \|X_0\| +(1+\iota_0) \|\bar F(X(t))\|+ G(t), \hspace{1cm} 0<t<\infty.
\end{align*}
The Lebesgue dominate theorem applied to \eqref{Ton15} then provides that 
\begin{equation*}    
X(t)=X(\epsilon) +\int_\epsilon^t [\bar F(X(u))-AX(u)]du, \hspace{1cm} \epsilon \leq t<\infty.
\end{equation*}
Therefore, $X$ is differentiable in $[\epsilon, \infty)$. Since $\epsilon>0$ is arbitrary,
$$X\in \mathcal C^1((0,\infty);E),$$
and $X$ satisfies the equation \eqref{Ton11}. This means that $X$ is a strong solution of \eqref{Ton11}. The proposition has thus been proved.
\end{proof}

\begin{remark}  
Proposition  \ref{THM5} deals with global solutions of \eqref{Ton11}. It is clear that the proposition is still true for local solutions. In other words, let $X$ be a local mild solution of \eqref{Ton11}, i.e. $X$ satisfies the conditions in Definition \ref{THM4} on some interval $[0,T_{loc}]$. If the function $G$ defined in Proposition  \ref{THM5} is integrable on any closed interval of  $[0,T_{loc}]$, then $X$ becomes a strong solution on the same interval.
\end{remark}

\section{Nonnegative local solutions} \label{local solutions}
In this section, we show existence and uniqueness of  nonnegative local strong solutions to   \eqref{Ton11}.

Let us consider   \eqref{Ton11}  with initial condition $X(t_0)=\zeta \in V$ instead of the condition $X(0)=X_0 \in V,$ where $t_0$ is some  nonnegative constant. In other words, we consider the equation
\begin{equation} \label{Ton18}
\begin{cases}
\frac{dX}{dt}+AX=\bar F(X), \hspace{1cm} t_0< t<\infty,\\
X(t_0)=\zeta  \in V
\end{cases}
\end{equation}
 in $E$. (We do this because results related to \eqref{Ton18} in this section are used in next sections.)

\begin{theorem}  \label{THM6}
There exists a unique local strong solution $X$ on $[t_0, t_0+\tau(\|\zeta\|)] $ to \eqref{Ton18}, where $\tau\colon [0,\infty)\to (0,\infty)$ is some continuous function, which is independent of $t_0$.   Furthermore, for any $\frac{3}{4}<\eta<1$,
\begin{equation}  \label{Ton19}
\|A^\eta X(t)\| \leq  [1+(\iota_0+\iota_\eta) \|\zeta\| ]  (t-t_0)^{-\eta}, \hspace{1cm} t_0<t\leq  t_0+\tau(\|\zeta\|),
\end{equation}
and 
\begin{equation}  \label{Ton20}
\|A X(t)\| \leq \chi(\|\zeta\|)  [1+ (t-t_0)^{-1}], \hspace{1cm} t_0<t\leq  t_0+\tau(\|\zeta\|),
\end{equation}
where $\chi$ is some real-valued positive continuous function on $[0,\infty)$, and is independent of $t_0$.
\end{theorem}

\begin{proof}
We  use the fixed point theorem for contractions to prove existence and uniqueness of  local mild solutions. We then use Proposition \ref{THM5} to show that the local mild solution is a strong solution.

Let's fix $\frac{3}{4}<\eta<1$ and $t_0<T<\infty$.  Set the underlying space
\begin{align*}
\Xi (T)=&\{Y\in \mathcal C((t_0,T];\mathcal D(A^\eta)) \cap  \mathcal C([t_0,T];E) \text{  such that   }\\
& \sup_{t_0<t\leq T} (t-t_0)^\eta \|A^\eta Y(t)\| <\infty \}.
\end{align*}
Then, $\Xi (T)$  becomes a Banach space with norm
\begin{equation*} 
\|Y\|_{\Xi (T)}=\sup_{t_0<t\leq T} (t-t_0)^\eta \|A^\eta Y(t)\|+ \sup_{t_0\leq t\leq T} \|Y(t)\|. 
\end{equation*}

Consider a subset $\Upsilon(T) $ of $\Xi (T)$ which consists of  functions $Y\in \Xi (T)$  such that
\begin{equation}  \label{Ton21}
\max\left\{\sup_{t_0<t\leq T} (t-t_0)^\eta \|A^\eta Y(t)\|,  \sup_{t_0\leq t\leq T} \|Y(t)\|\right\} \leq C_1,
\end{equation}
where 
\begin{equation*}
C_1= 1+(\iota_0+\iota_\eta)\|\zeta\|.
\end{equation*}
Obviously, $\Upsilon(T) $  is a nonempty closed subset of $\Xi (T)$.

For $Y\in \Upsilon(T)$, we define a function on $[t_0,T]$ by
\begin{align*}  
\Phi Y(t)=&e^{-(t-t_0)A} \zeta+\int_{t_0}^t e^{-(t-s)A}[\bar F(Y(s))ds. 
\end{align*}
Our goal is then to verify that $\Phi$ is a contraction mapping from $\Upsilon(T)$ into itself, provided  $T$ is sufficiently small, and that the fixed point of $\Phi$ is the desired solution of  \eqref{Ton18}. For this purpose, we divide the proof into six steps.

{\bf Step 1}. Let us show that $ \Phi Y \in \Upsilon(T)$ for  $Y \in \Upsilon(T). $

Thanks to \eqref{Ton5} and Lemma \ref{THM3},
\begin{align*}
\|\Phi Y(t)\| \leq & \|e^{-(t-t_0)A}\zeta\|+ \int_{t_0}^t \|e^{-(t-s)A}\| \|\bar F(Y(s))\|ds \\
\leq & \iota_0 \|\zeta\| +\iota_0 c_\infty \int_{t_0}^t  (1+\|Y(s)\|)ds  \\
\leq & \iota_0 \|\zeta\| +\iota_0 c_\infty (1+C_1) (t-t_0), \hspace{1cm} t_0\leq t\leq T.
\end{align*}
Furthermore, \eqref{Ton4} gives 
\begin{align*}
\|A^\eta \Phi Y(t)\| \leq & \|A^\eta  e^{-(t-t_0)A}\zeta\|+ \int_{t_0}^t \|A^\eta  e^{-(t-s)A}\| \|\bar F(Y(s))\|ds \\
\leq & \iota_\eta (t-t_0)^{-\eta}  \|\zeta\| +\iota_\eta c_\infty \int_{t_0}^t (t-s)^{-\eta}  [1+\|Y(s)\|]ds  \\
\leq & \iota_\eta (t-t_0)^{-\eta}  \|\zeta\| +\iota_\eta c_\infty (1+C_1) \int_{t_0}^t (t-s)^{-\eta}ds \\
=&\iota_\eta (t-t_0)^{-\eta}  \|\zeta\| +\frac{\iota_\eta c_\infty (1+C_1)  (t-t_0)^{1-\eta}   }{1-\eta}    \hspace{1cm} t_0< t\leq T.
\end{align*}

Therefore, it is easily seen that 
\begin{equation*}  
\max\left\{\sup_{t_0<t\leq T} (t-t_0)^\eta \|A^\eta \Phi Y(t)\|,  \sup_{t_0\leq t\leq T} \|\Phi Y(t)\|\right\} \leq C_1,
\end{equation*}
provided  $T$ is sufficiently small. 
In fact,  $T$ can be chosen to be any function of $\|\zeta\|$ such that
\begin{equation}   \label{Ton22}
T-t_0\leq \min\Big\{\frac{1}{\iota_0},\frac{1}{\iota_\eta}\Big\}  
\frac{1+[\iota_\eta+\iota_0(1-\eta)]\|\zeta\|}{c_\infty [2+(\iota_0+\iota_\eta)\|\zeta\|]}.
\end{equation}

It now remains to prove the continuity of $\Phi Y$ and $A^\eta \Phi Y$ on $[t_0,T]$ and $(t_0,T],$ respectively. For $t_0<s<t\leq T$, the semigroup property gives 
\begin{align*}
\Phi Y(t)=& e^{-(t-s)A} e^{-(s-t_0)A}\zeta+\int_s^t e^{-(t-u)A} \bar F(Y(u))du \\
& +e^{-(t-s)A} \int_{t_0}^s e^{-(s-u)A} \bar F(Y(u))du.
\end{align*}
Using this equality, \eqref{Ton4} and Lemma \ref{THM3}, we have
\begin{align*}
& \|A^\eta \Phi Y(t)-A^\eta \Phi Y(s)\|\\
 \leq  & \|A^\eta  e^{-(s-t_0)A}[e^{-(t-s)A}-I] \zeta\|  +\int_s^t \|A^\eta e^{-(t-u)A}\| \| \bar F(Y(u))\|du\\
&+\|e^{-(t-s)A}-I\|  \int_{t_0}^s \|A^\eta e^{-(s-u)A}\| \| \bar F(Y(u))\|du\\
 \leq  &  \iota_\eta \|\zeta\| (s-t_0)^{-\eta} \|e^{-(t-s)A}-I\|
+\iota_\eta c_\infty \int_s^t (t-u)^{-\eta} (1+\|Y(u)\|) du\\
&+\iota_\eta  c_\infty\|e^{-(t-s)A}-I\|  \int_{t_0}^s (s-u)^{-\eta}
(1+\|Y(u)\|) du \\
 \leq  &  \iota_\eta \|\zeta\| (s-t_0)^{-\eta} \|e^{-(t-s)A}-I\|
+\iota_\eta c_\infty (1+C_1) \int_s^t (t-u)^{-\eta} du\\
&+\iota_\eta  c_\infty (1+C_1) \|e^{-(t-s)A}-I\|  \int_{t_0}^s (s-u)^{-\eta} du\\
 =  &  \iota_\eta \|\zeta\| (s-t_0)^{-\eta} \|e^{-(t-s)A}-I\|
+\frac{\iota_\eta c_\infty (1+C_1)  (t-s)^{1-\eta} }{1-\eta}  \\
&+\frac{\iota_\eta  c_\infty (1+C_1)  (s-t_0)^{1-\eta}  \|e^{-(t-s)A}-I\| }{1-\eta}.
\end{align*}
Since 
$$\lim_{t\to s} \|e^{-(t-s)A}-I\|=0,$$
we observe that 
$$\lim_{t\to s} \|A^\eta \Phi Y(t)-A^\eta \Phi Y(s)\|=0.$$
Thus,   $A^\eta \Phi Y$ is continuous on $(t_0,T].$ As a consequence, $ \Phi Y$ is also continuous on $(t_0,T]$ because
$A^{-\eta}\in L(E)$ and
$\Phi Y=A^{-\eta} (A^\eta \Phi).$

We finally show the continuity of $ \Phi Y$ at $t=t_0$. Due to \eqref{Ton5}, \eqref{Ton21}, and Lemma \ref{THM3}, we have
\begin{align*}
\|\Phi Y(t)-\Phi Y(t_0)\|=& \|[e^{-(t-t_0)A}-I]\zeta+\int_{t_0}^t e^{-(t-u)A} \bar F(Y(u))du \|   \\
\leq & \|[e^{-(t-t_0)A}-I]\zeta\|+\iota_0 c_\infty \int_{t_0}^t [1+\|Y(u)\|]du   \\
\leq & \|[e^{-(t-t_0)A}-I]\zeta\|+\iota_0 c_\infty (1+C_1) (t-t_0).
\end{align*}
Thus,
$$\lim_{t\to t_0} \Phi Y(t)=\Phi Y(t_0).$$

{\bf Step 2}. Let us show that $\Phi$ is a contraction mapping on $\Upsilon(T),$ provided  $T$ is sufficiently small. 

Let $Y_1, Y_2 \in \Upsilon(T).$ Lemma \ref{THM3} and \eqref{Ton5}  give
\begin{align*}
&\|\Phi Y_1(t)-\Phi Y_2(t)\|\\
&=\|\int_{t_0}^t e^{-(t-u)A} [\bar F(Y_1(u))-\bar F(Y_2(u))] du \|  \\
 &\leq \iota_0 c_\infty \int_{t_0}^t  [1+\|Y_1(s)\|_{\infty} +\|Y_2(s)\|_{\infty}]  \| Y_1(u)-Y_2(u)\| du, \hspace{0.5cm} t_0\leq t\leq T. 
\end{align*}
Since
$\mathcal D(A^\eta)=H^{2\eta}(\Omega)\times H^{2\eta}(\Omega) $
and $H^{2\eta}(\Omega) \subset \mathcal C(\Omega) $ with continuous embedding, there exists a constant $C$ such that 
\begin{equation}    \label{Ton23}
\|Y_i(s)\|_{\infty} \leq C\|A^\eta Y_i(s)\|, \hspace{1cm} t_0<s\leq T, i=1,2.
\end{equation}
Thus,
\begin{align}
&\|\Phi Y_1(t)-\Phi Y_2(t)\|     \notag\\
 &\leq \iota_0 c_\infty \int_{t_0}^t  [1+C\|A^\eta Y_1(u)\| +C\|A^\eta Y_2(u)\|]  \| Y_1(u)-Y_2(u)\| du     \notag\\
 &\leq \iota_0 c_\infty \int_{t_0}^t  [1+2CC_1 (u-t_0)^{-\eta}] \| Y_1(u)-Y_2(u)\| du  \hspace{1cm} (\text{see } \eqref{Ton21})       \notag\\
&\leq \iota_0 c_\infty \int_{t_0}^t  [1+2CC_1 (u-t_0)^{-\eta}] du \| Y_1-Y_2\|_{\Xi (T)}     \notag\\
&= \iota_0 c_\infty   \Big[t-t_0+\frac{2CC_1 (t-t_0)^{1-\eta})} {1-\eta} \Big]  \| Y_1-Y_2\|_{\Xi (T)}, \hspace{0.5cm} t_0\leq t\leq T.   \label{Ton24}
\end{align}

Similarly, we have for $t_0\leq t\leq T$, 
\begin{align}
(t-t_0)^\eta &\|A^\eta[\Phi Y_1(t)-\Phi Y_2(t)]\|     \notag\\
&=(t-t_0)^\eta \|\int_{t_0}^t A^\eta e^{-(t-u)A} [\bar F(Y_1(u))-\bar F(Y_2(u))] du \|       \notag\\
 &\leq \iota_\eta c_\infty (t-t_0)^\eta  \int_{t_0}^t (t-u)^{-\eta} [1+C\|A^\eta Y_1(u)\| +C\|A^\eta Y_2(u)\|]    \notag\\
& \hspace{2cm} \times \| Y_1(u)-Y_2(u)\| du     \notag\\
 &\leq \iota_\eta c_\infty (t-t_0)^\eta \int_{t_0}^t (t-u)^{-\eta} [1+2CC_1 (u-t_0)^{-\eta}] \notag\\
& \hspace{2cm} \times  \| Y_1(u)-Y_2(u)\| du        \notag\\
&\leq \iota_\eta c_\infty (t-t_0)^\eta  \int_{t_0}^t  (t-u)^{-\eta} [1+2CC_1 (u-t_0)^{-\eta}] du \| Y_1-Y_2\|_{\Xi (T)}        \notag\\
&= \iota_\eta c_\infty   \Big[\frac{t-t_0} {1-\eta} +2CC_1 B(1-\eta,1-\eta)(t-t_0)^{1-\eta} \Big]  \| Y_1-Y_2\|_{\Xi (T)},   \label{Ton25}
\end{align}
 where $B(\cdot,\cdot)$ is the Beta function.

Combining the estimates \eqref{Ton24} and \eqref{Ton25}, we obtain that 
\begin{align*}
&\| \Phi Y_1-\Phi Y_2\|_{\Xi (T)} \\
\leq & c_\infty\Big[(\iota_0+\frac{\iota_\eta}{1-\eta})(T-t_0)+2CC_1\{\frac{\iota_0}{1-\eta}+\iota_\eta B(1-\eta,1-\eta)\}(T-t_0)^{1-\eta}\Big]  \\
& \times \| Y_1-Y_2\|_{\Xi (T)}.
\end{align*}
This shows that $\Phi$ is contraction on $\Upsilon(T),$ provided  $T$ is sufficiently small. In fact, $T$ can be chosen to be any function of $\|\zeta\|$ such that
\begin{equation}  \label{Ton26}
(\iota_0+\frac{\iota_\eta}{1-\eta})(T-t_0)+2CC_1[\frac{\iota_0}{1-\eta}+\iota_\eta B(1-\eta,1-\eta)](T-t_0)^{1-\eta} <\frac{1}{c_\infty}.
\end{equation}

{\bf Step 3.} Let us prove existence of a local mild solution to  \eqref{Ton18}.

Let $T>t_0$ be sufficiently small in such a way that $\Phi$ maps $\Upsilon(T)$ into itself and is contraction with respect to the norm of $\Xi (T)$. Because of  \eqref{Ton22} and \eqref{Ton26}, we can choose $T=t_0+\tau(\|\zeta\|)$, where $\tau(\|\zeta\|)$ continuously depends only on $\|\zeta\|$.  Thanks to the fixed point theorem, there exists a unique function $X\in \Upsilon(t_0+\tau(\|\zeta\|))$ such that $X(t)=\Phi X(t)$ for $t_0\leq t\leq t_0+\tau(\|\zeta\|).$ This means that $X$ is a local mild solution of  \eqref{Ton18} on $[t_0,t_0+\tau(\|\zeta\|)]$. 

{\bf Step 4.} Let us prove uniqueness of  local mild solutions to  \eqref{Ton18}.

Suppose that $\bar X$ is any other  local mild solution of  \eqref{Ton18} on $[t_0,t_0+\tau(\|\zeta\|)]$.  The solution formula of $\bar X$, \eqref{Ton5}, and Lemma \ref{THM3}  then imply that
\begin{align*}
\|\bar X(t)\| \leq & \|e^{-(t-t_0)A}\zeta\|+ \int_{t_0}^t \|e^{-(t-s)A} \bar F(\bar X(s))\| ds  \\
\leq & \iota_0 \|\zeta\| + \iota_0 c_\infty \int_{t_0}^t [1+\|\bar X(s)\|]ds, \hspace{1cm} t_0\leq t\leq t_0+\tau(\|\zeta\|).
\end{align*}
The Gronwall inequality then provides that 
\begin{align*}
1+\|\bar X(t)\| & \leq (1+\iota_0 \|\zeta\|) e^{\iota_0c_\infty (t-t_0)}  \\
&\leq 
(1+\iota_0 \|\zeta\|) e^{\iota_0c_\infty \tau(\|\zeta\|)},
 \hspace{1cm} t_0\leq t\leq t_0+\tau(\|\zeta\|).
\end{align*}
Using this, we have for $ t_0< t\leq t_0+\tau(\|\zeta\|)$, 
\begin{align*}
& \|A^\eta \bar X(t)\| \\
\leq & \|A^\eta e^{-(t-t_0)A}\zeta\|+ \int_{t_0}^t \|A^\eta e^{-(t-s)A} \bar F(\bar X(s))\| ds  \\
\leq & \iota_\eta \|\zeta\|(t-t_0)^{-\eta} + \iota_\eta c_\infty \int_{t_0}^t (t-s)^{-\eta} [1+\|\bar X(s)\|]ds\\
\leq & \iota_\eta \|\zeta\|(t-t_0)^{-\eta} + \iota_\eta c_\infty (1+\iota_0 \|\zeta\|) e^{\iota_0c_\infty \tau(\|\zeta\|)} \int_{t_0}^t (t-s)^{-\eta} ds\\
=& \iota_\eta \|\zeta\|(t-t_0)^{-\eta} + \frac{\iota_\eta c_\infty (1+\iota_0 \|\zeta\|) e^{\iota_0c_\infty \tau(\|\zeta\|)}  (t-t_0)^{1-\eta}} {1-\eta}.
\end{align*}
Hence,
\begin{align*}
\|A^\eta \bar X(t)\| 
&\leq C(\|\zeta\|) [(t-t_0)^{-\eta}+(t-t_0)^{1-\eta}],  \hspace{1cm} t_0< t\leq t_0+\tau(\|\zeta\|),
\end{align*}
where $C(\|\zeta\|) $ is a function of  $\|\zeta\|$ and is independent of $t_0$.

Let us estimate the difference between $X$ and $\bar X$. We observe from the latter inequality, Lemma \ref{THM3},  \eqref{Ton5}, and \eqref{Ton21} that 
\begin{align*}
&\|X(t)-\bar X(t)\|\\
&= \|\int_{t_0}^t e^{-(t-s)A} [\bar F(X(s))-\bar F(\bar X(s)) ]ds\|\\
 & \leq\iota_0 c_\infty  \int_{t_0}^t [1+\|X(s)\|_{\infty}+\|\bar X(s)\|_{\infty}] \|X(s)-\bar X(s)\|ds\\
 &\leq \iota_0 c_\infty  \int_{t_0}^t [1+\|A^\eta X(s)\|+\|A^\eta \bar X(s)\|] \|X(s)-\bar X(s)\|ds \\
& \leq  \iota_0 c_\infty  \int_{t_0}^t [1+C_1 (s-t_0)^{-\eta}  +C(\|\zeta\|) \{(s-t_0)^{-\eta}+(s-t_0)^{1-\eta}\}]  \\
&  \hspace{2cm}\times \|X(s)-\bar X(s)\|ds,
 \hspace{1cm} t_0\leq t\leq t_0+\tau(\|\zeta\|).
\end{align*}

Let $t_0\leq S \leq t_0+\tau(\|\zeta\|)$. Taking supremum on $[t_0,S]$ in both the hand sides of the above estimate, we arrive at
\begin{align}    
&\sup_{t\in [t_0,S]}\|X(t)-\bar X(t)\|   \notag\\
& \leq  \iota_0  \sup_{t\in [t_0,S]} \int_{t_0}^t [1+C_1 (s-t_0)^{-\eta}  +C(\|\zeta\|) \notag\\
& \hspace{1cm} \times  \{(s-t_0)^{-\eta}+(s-t_0)^{1-\eta}\}] ds   \sup_{s\in [t_0,S]}\|X(s)-\bar X(s)\|   \notag\\
&=\iota_0  \Big[S-t_0+\frac{\{C_1+C(\|\zeta\|)\} (S-t_0)^{1-\eta}}{1-\eta}  +\frac{ C(\|\zeta\|)(S-t_0)^{2-\eta}}{2-\eta}\Big]    \label{Ton27}\\
& \hspace{1cm} \times  \sup_{t\in [t_0,S]}\|X(t)-\bar X(t)\|.   \notag
\end{align}
We choose $t_0\leq S \leq t_0+\tau(\|\zeta\|)$ such that 
$$\iota_0  \Big[S-t_0+\frac{\{C_1+C(\|\zeta\|)\} (S-t_0)^{1-\eta}}{1-\eta}  +\frac{ C(\|\zeta\|)(S-t_0)^{2-\eta}}{2-\eta}\Big] <1.$$
Then, \eqref{Ton27} gives 
$$ \sup_{t\in [t_0,S]}\|X(t)-\bar X(t)\|=0,$$
that is
$$X(t)=\bar X(t), \hspace{1cm}   t_0\leq t\leq S.$$

Let $\tilde S=\sup\{t_0\leq S \leq t_0+\tau(\|\zeta\|); X(t)=\bar X(t) \text{ for all } t_0\leq t\leq S\}.$ By continuity, $X(\tilde S)=\bar X(\tilde S).$ Suppose that $\tilde S<t_0+\tau(\|\zeta\|)$. Repeating the same procedure with initial time $\tilde S$ and initial value $X(\tilde S)=\bar X(\tilde S),$  we obtain that for every sufficiently small $t>0,$
$$X(\tilde S+t)=\bar X(\tilde S+t).$$
This contradicts the definition of $\tilde S$. Thus,
$$X(t)=\bar X(t), \hspace{1cm}   t_0\leq t\leq t_0+\tau(\|\zeta\|).$$

{\bf Step 5.} Let us prove that the local mild solution $X$  on $[t_0,t_0+\tau(\|\zeta\|)]$ is a local strong solution to  \eqref{Ton18}. For the proof, we use Proposition \ref{THM5}--(ii).

First, we give an estimate for $X(t)-X(s), t_0 <s<t\leq t_0+\tau(\|\zeta\|).$ Let $0<\epsilon<1-\eta$. We have
\begin{align*}
& X(t)-X(s) \\
= & [e^{-(t-t_0)A}-e^{-(s-t_0)A}]\zeta +\int_s^t e^{-(t-u)A} \bar F(X(u)) du  \\
&+ \int_{t_0}^s [e^{-(t-u)A}-e^{-(s-u)A}] \bar F(X(u)) du   \\
= & [e^{-(t-s)A}-I] A^{-\epsilon} A^\epsilon e^{-(s-t_0)A}\zeta +\int_s^t e^{-(t-u)A} \bar F(X(u)) du  \\
&+ [e^{-(t-s)A}-I] A^{-\epsilon}\int_{t_0}^s A^\epsilon e^{-(s-u)A} \bar F(X(u)) du. 
\end{align*}
Then, \eqref{Ton4}, \eqref{Ton5}, \eqref{Ton6}, and Lemma \ref{THM3} give
\begin{align*}
& \|X(t)-X(s)\| \\
%
%
\leq  & \|[e^{-(t-s)A}-I] A^{-\epsilon}\| \| A^\epsilon e^{-(s-t_0)A}\| \|\zeta\| +\int_s^t \|e^{-(t-u)A}\| \|\bar F(X(u))\| du  \\
&+ \|[e^{-(t-s)A}-I] A^{-\epsilon}\| \int_{t_0}^s \|A^\epsilon e^{-(s-u)A}\| \| \bar F(X(u))\| du \\
\leq & \frac{(1-\epsilon)\iota_\epsilon \|\zeta\|}{\epsilon} (t-s)^\epsilon (s-t_0)^{-\epsilon}
+\iota_0 c_\infty \int_s^t [1+\|X(u)\|] du  \\
&+\frac{(1-\epsilon) \iota_\epsilon  c_\infty}{\epsilon}  (t-s)^\epsilon  \int_{t_0}^s  (s-u)^{-\epsilon} [1+\|X(u)\|] du.
\end{align*}
Since $X\in \Upsilon(t_0+\tau(\|\zeta\|))$, $X$ satisfies \eqref{Ton21}. Thus,
\begin{align}
& \|X(t)-X(s)\|    \notag \\
%
\leq & \frac{(1-\epsilon)\iota_\epsilon \|\zeta\|}{\epsilon} (t-s)^\epsilon (s-t_0)^{-\epsilon}
+\iota_0 c_\infty  (1+C_1) (t-s)     \notag\\
&+\frac{(1-\epsilon) (1+C_1)\iota_\epsilon  c_\infty}{\epsilon}  (t-s)^\epsilon  \int_{t_0}^s  (s-u)^{-\epsilon} du     \notag\\
=& \frac{(1-\epsilon)\iota_\epsilon \|\zeta\|}{\epsilon} (t-s)^\epsilon (s-t_0)^{-\epsilon}
+\iota_0 c_\infty  (1+C_1) (t-s)     \notag\\
&+\frac{ (1+C_1)\iota_\epsilon  c_\infty}{\epsilon}  (t-s)^\epsilon    (s-t_0)^{1-\epsilon}    \notag\\
\leq & C_2(\|\zeta\|) [(t-s) + (t-s)^\epsilon (s-t_0)^{-\epsilon} + (t-s)^\epsilon    (s-t_0)^{1-\epsilon}]     \label{Ton28}
\end{align}
for every $t_0 <s<t\leq t_0+\tau(\|\zeta\|) $. Here,  
 $C_2(\|\zeta\|)$ is some polynomial of first degree of $\|\zeta\|$ and is independent of $t_0$.

We now check the condition in Proposition \ref{THM5}--(ii). We have to show that for  $t_0\leq S\leq t_0+\tau(\|\zeta\|)$, 
\begin{align}    \label{Ton29}
\int_{t_0}^S \int_{t_0}^t (t-s)^{-1} \|\bar F(X(t))-\bar F(X(s))\| ds<\infty.
\end{align}
Indeed, for $t_0< s, t\leq t_0+\tau(\|\zeta\|)$, Lemma \ref{THM3},  \eqref{Ton21} and \eqref{Ton23} give
\begin{align}
& \|\bar F(X(t))-\bar F(X(s))\|    \notag\\
\leq & c_\infty [1+\|X(t)\|_{L_{\infty}}+\|X(s)\|_{L_{\infty}}]  \|X(t)-X(s)\|   \notag\\
\leq &c_\infty [1+C \|A^\eta X(t)\|+C\|A^\eta X(s)\|]  \|X(t)-X(s)\|   \notag\\
\leq &c_\infty [1+C C_1 (t-t_0)^{-\eta} +C C_1 (s-t_0)^{-\eta}]  \|X(t)-X(s)\|.   \label{Ton30}
\end{align}
By some simple calculations, it is easily seen that \eqref{Ton29} follows from  \eqref{Ton28} and \eqref{Ton30}.
We then conclude that $X$  is a local strong solution  on $[t_0,t_0+\tau(\|\zeta\|)]$ to  \eqref{Ton18}.

{\bf Step 6.}  Let us finally prove the estimates \eqref{Ton19} and  \eqref{Ton20}.  The first one is obvious due to \eqref{Ton21}.  For the second one, we have
\begin{align*}
AX(t)=& A e^{-(t-t_0)A} \zeta +\int_{t_0}^t A e^{-(t-u)A} \bar F(X(u)) du \notag\\
=&  A e^{-(t-t_0)A} \zeta +\int_{t_0}^t A e^{-(t-u)A} [\bar F(X(u))-\bar F(X(t))] du  \notag\\
&+\int_{t_0}^t A e^{-(t-u)A} du \bar F(X(t))     \notag \\
=&   A e^{-(t-t_0)A}  \zeta +\int_{t_0}^t A e^{-(t-u)A} [\bar F(X(u))-\bar F(X(t))] du    \\
&+[I-e^{-(t-t_0)A}]\bar F(X(t)), \hspace{1cm} t_0<t\leq t_0+\tau(\|\zeta\|).  \notag
\end{align*}
Using \eqref{Ton4}, \eqref{Ton5}, and Lemma \ref{THM3}, it follows that
\begin{align*}
\|AX(t)\| 
\leq &  \iota_1 (t-t_0)^{-1} \|\zeta\|  \\
& +\iota_1 c_\infty \int_{t_0}^t (t-u)^{-1}  [1+\|X(u)\|_{\infty} + \|X(t)\|_{\infty} ] \\
& \hspace{2cm} \times  \| X(t)-X(u)\| du    \\
&+c_\infty  (1+\iota_0) [1+\|X(t)\|], \hspace{1cm} t_0<t\leq t_0+\tau(\|\zeta\|).  \notag
\end{align*}
Hence,  \eqref{Ton20},  \eqref{Ton21}, and  \eqref{Ton23}  give
\begin{align*}
&\|AX(t)\| \\
\leq &  \iota_1 (t-t_0)^{-1} \|\zeta\|  \\
& +\iota_1 c_\infty \int_{t_0}^t (t-u)^{-1}  [1+C\|A^\eta X(u)\| + C\|A^\eta  X(t)\| ] \\
& \hspace{2cm} \times  \| X(t)-X(u)\| du    \\
&+c_\infty  (1+\iota_0) [2+(\iota_0+\iota_\eta)\|\zeta\|]\\
\leq &  \iota_1 (t-t_0)^{-1} \|\zeta\|  \\
& +\iota_1 c_\infty \int_{t_0}^t (t-u)^{-1}  [1+C\{1+(\iota_0+\iota_\eta)\|\zeta\|\}  \{(u-t_0)^{-\eta}+(t-t_0)^{-\eta}\} ] \\
& \hspace{2cm} \times  \| X(t)-X(u)\| du    \\
&+c_\infty  (1+\iota_0) [2+(\iota_0+\iota_\eta)\|\zeta\|], \hspace{1cm} t_0<t<t_0+\tau(\|\zeta\|).  \notag
\end{align*}
Thus, there exists a constant, still denoted by $C>0,$ depending only on the exponents such that for $t_0<t\leq t_0+\tau(\|\zeta\|),$ 
\begin{align}
\|AX(t)\|  
\leq  & C (1+\|\zeta\|) + C (t-t_0)^{-1} \|\zeta\|       \label{Ton31}\\
& +C (1+\|\zeta\|) \int_{t_0}^t (t-u)^{-1}    [1+(u-t_0)^{-\eta}+(t-t_0)^{-\eta}]   \notag\\
& \hspace{3cm} \times   \| X(t)-X(u)\| du.      \notag
\end{align}
By substituting \eqref{Ton28} (change $s$ by $u$) into \eqref{Ton31} and some simple calculations, there exists a real-valued positive continuous function $\chi$ on $[0,\infty)$, which is independent of $t_0$, such that 
\begin{equation*}  
\|A X(t)\| \leq \chi(\|\zeta\|)  [1+ (t-t_0)^{-1}], \hspace{1cm} t_0<t\leq  t_0+\tau(\|\zeta\|).
\end{equation*}
Thus, \eqref{Ton20} has been proved. The proof  of the theorem is complete.
\end{proof}

\begin{theorem}  \label{THM7}
The two components of  the local strong solution $X$ on $[t_0, t_0+\tau(\|\zeta\|)] $ in Theorem \ref{THM6}  are real-valued and nonnegative.
\end{theorem}

\begin{proof}
It is clear that the function $\bar X$ defined by the complex conjugate of $ X$  is also a solution on $[t_0, t_0+\tau(\|\zeta\|)] $ of \eqref{Ton18} with the same initial value $\zeta$. The uniqueness of solution then implies that 
$$X(t)=\overline {X(t)}, \hspace{1cm} t_0\leq t\leq t_0+\tau(\|\zeta\|).$$
Hence, $X$ is real-valued.

In order to prove the nonnegativity of the two components of $X=
\left (\begin{matrix}
u\\
v
\end{matrix}\right)
$, we use a cutoff function $H$ given by
\begin{equation*}
H(x)=
\begin{cases}
\frac{1}{2} x^2, \hspace{1cm} -\infty<x<0,\\
0, \hspace{1cm}  0\leq x<\infty.
\end{cases}
\end{equation*} 
The function $H$ has the following property. 
For any $f\in \mathcal C^1((0,\infty);L_2(\Omega )),$ the function $G_f$ defined by
$$G_f(t)=\int_\Omega  H((f(t))dx, \hspace{1cm} 0<t<\infty,$$
has continuous derivative
$$G_f'(t)=\int_\Omega  H'(f(t)) f'(t) dx, \hspace{1cm} 0<t<\infty. $$
(See \cite{yagi} for the proof.) Thus, the function $G_u$ is continuously differentiable with the derivative
\begin{align*}
G_u'(t)=&\int_\Omega  H'(u) \Delta u dx  + \gamma \int_\Omega  H'(u) [ a-u-\frac{ \rho u|v|}{1+|u|+ku^2}]dx. 
\end{align*}
We have
\begin{align*}
\int_\Omega  H'(u) \Delta u dx =&-\int_\Omega  \nabla  H'(u) \cdot \nabla u dx   \\
=&-\int_\Omega  |\nabla u |^2 dx \leq 0.
\end{align*}
Furthermore, since $H'(x)\leq 0  $ and $H'(x)x \geq 0$ for every $x\in \mathbb R$, 
$$\int_\Omega  H'(u) [ a-u-\frac{ \rho u|v|}{1+|u|+ku^2}]dx  \leq 0.$$
Hence,
$$G_u'(t) \leq 0, \hspace{1cm} t_0\leq t\leq t_0+\tau(\|\zeta\|).$$
As a consequence,
$$G_u(t) \leq G_u(t_0), \hspace{1cm} t_0\leq t\leq t_0+\tau(\|\zeta\|). $$

Because $H$ is nonnegative  and $G_u(t_0)=0$, we obtain that
$$G_u(t)=\int_\Omega  H((u(t))dx =0 \hspace{1cm} \text{ for all }  t_0\leq t\leq t_0+\tau(\|\zeta\|).$$
Thus,
$$u(t) \geq 0 \hspace{1cm} \text{ for all }  t_0\leq t\leq \tau(\|\zeta\|).$$

Similarly, the function $G_v$ is continuously differentiable with the derivative
\begin{align*}
G_v'(t)=&\alpha  \int_\Omega  H'(v) \Delta v dx  + \gamma \int_\Omega  H'(v) [ \beta(b-v)-\frac{ \rho |u|v}{1+|u|+ku^2}]dx. 
\end{align*}
Using the same arguments as above, we obtain that
$$v(t) \geq 0 \hspace{1cm} \text{ for all }  t_0\leq t\leq t_0+\tau(\|\zeta\|).$$
The proof is complete.
\end{proof}

\begin{remark}
Thanks to Theorem \ref{THM7}, a strong solution of the system  \eqref{Ton11} is also a solution of the system \eqref{Ton1} coupled with the Neumann boundary condition \eqref{Ton2}.
\end{remark}

\section{Global solutions}  \label{global solutions}
In this section, we show that the system 
\eqref{Ton1} coupled with \eqref{Ton2}, or equivalently the system  \eqref{Ton11} possesses a unique global positive solution. For this purpose, we use a {\it priori} estimate for solutions.

\begin{theorem}[{\em priori} estimate] \label{THM8}
The  strong solution  $X$ in Theorem  \ref{THM6}  satisfies a  norm estimate
\begin{equation} \label{Ton32}
\|X(t)\| \leq C[e^{-\mu (t-t_0)} \|\zeta\| +1], \hspace{1cm} t_0<t\leq t_0+ \tau(\|\zeta\|)
\end{equation}
with some constants  $0<\mu, C<\infty$ independent of $t_0$ and $\zeta.$ As a consequence, 
\begin{equation} \label{Ton33}
\|X(t)\| \leq C(\|\zeta\| +1), \hspace{1cm} t_0\leq t\leq t_0+\tau(\|\zeta\|).
\end{equation}
\end{theorem}

\begin{proof}
We have $X=
\left (\begin{matrix}
u\\
v
\end{matrix}\right)
$, where $(u,v)$ satisfies  \eqref{Ton1} on $(t_0, t_0+ \tau(\|\zeta\|)]$ with the Neumann boundary condition and 
$\left (\begin{matrix}
u(t_0)\\
v(t_0)
\end{matrix}\right)
=
\zeta.$  
Consider the inner product of the two equations in \eqref{Ton1} and $u, v$ in 
$L_2(\Omega )$, respectively. From the equation on $u$, we have
\begin{align*}
\int_{\Omega } \frac{u\partial u }{\partial t}  dx=\int_{\Omega } u\Delta u  dx +\int_{\Omega } \gamma \Big(a-u-\frac{\rho uv}{1+u+ku^2}\Big) u  dx.
\end{align*}
Using the Neumann boundary condition \eqref{Ton2}, this equality implies that
\begin{align*}
\frac{1}{2} \frac{d}{dt} \int_{\Omega } u^2  dx +\int_{\Omega } |\nabla  u|^2  dx + \gamma \int_{\Omega } u^2 dx= \int_{\Omega } \gamma \Big(a-\frac{\rho uv}{1+u+ku^2}\Big) u  dx.
\end{align*}
Hence,
\begin{align}
\frac{1}{2} \frac{d}{dt} \int_{\Omega } u^2  dx  + \gamma \int_{\Omega } u^2 dx\leq  \int_{\Omega } \gamma \Big(a-\frac{\rho uv}{1+u+ku^2}\Big) u  dx.  \label{Ton34}
\end{align}

Similarly, from the equation on $v$, we have
\begin{align}
\frac{1}{2} \frac{d}{dt} \int_{\Omega } v^2  dx  + \gamma \beta\int_{\Omega } v^2 dx\leq  \int_{\Omega } \gamma \Big(\beta b-\frac{\rho uv}{1+u+ku^2}\Big) v  dx.  \label{Ton35}
\end{align}

Put $\gamma^*=\frac{\min\{\gamma, \gamma \beta\}}{2}$. Combining \eqref{Ton34} and \eqref{Ton35}, we obtain that
\begin{align}
&\frac{1}{2} \frac{d}{dt} \int_{\Omega } (u^2+v^2)  dx  + 2\gamma^*\int_{\Omega } (u^2+v^2) dx   \label{Ton36}\\
&\leq  \int_{\Omega } \gamma \Big(au+\beta b v-\frac{\rho u^2v}{1+u+ku^2}-\frac{\rho uv^2}{1+u+ku^2}\Big)   dx.   \notag
\end{align}
It is easily seen that there exists a constant $C>0$  such that for all $0\leq u,v<\infty$, 
$$\gamma \Big(au+\beta b v-\frac{\rho u^2v}{1+u+ku^2}-\frac{\rho uv^2}{1+u+ku^2}\Big)  \leq \gamma^* (u^2+v^2) +\frac{C}{\text {Vol}(\Omega)}.$$
We then observe from  \eqref{Ton36} that 
\begin{align*}
\frac{1}{2} \frac{d}{dt} \int_{\Omega } (u^2+v^2)  dx  + \gamma^*\int_{\Omega } (u^2+v^2) dx  
\leq  C, \hspace{0.5cm} t_0<t\leq t_0+\tau(\|\zeta\|).   
\end{align*}

Solving this differential inequality, we obtain that for $t_0<t\leq t_0+\tau(\|\zeta\|)$, 
\begin{align*}
&\int_{\Omega } (u^2+v^2)  dx \\
&\leq e^{-2\gamma^* (t-t_0)} \int_{\Omega } [u(t_0)^2+v(t_0)^2]  dx +C\int_{t_0}^t e^{-2\gamma^*(t-s)}ds \\
&=e^{-2\gamma^* (t-t_0)} \int_{\Omega } [u(t_0)^2+v(t_0)^2]  dx + \frac{C(1-e^{-2\gamma^* t)}}{2\gamma^*}.
\end{align*} 
Thus,
\begin{align*}
& \|u(t)\|_{L^2}^2+\|v(t)\|_{L^2}^2 \\
&\leq e^{-2\gamma^* (t-t_0)} [\|u(t_0)\|_{L^2}^2+\|v(t_0)\|_{L^2}^2]+\frac{C}{2\gamma^*},  \hspace{1cm} t_0<t\leq t_0+\tau(\|\zeta\|).
\end{align*} 
The estimates \eqref{Ton32} and \eqref{Ton33} then follow this estimate.  The theorem has been proved.
\end{proof}

Thanks to Theorems \ref{THM6} and \ref{THM7}, the system \eqref{Ton11} possesses a unique local positive strong solution $X$ on $[0, \tau(\|X_0\|)]$, where $\tau$ is the function defined in Theorem  \ref{THM6}.  We are now ready to show that $X$ is  defined globally.

\begin{theorem}  \label{THM9}
The solution $X$ of the system  \eqref{Ton11} 
can be prolonged to $[0,\infty)$. In other words,  \eqref{Ton11} possesses a unique global positive strong solution $X$  satisfying a  global norm estimate
\begin{equation}  \label{Ton37}  
\|X(t)\| \leq C(e^{-\mu t} \|X_0\| +1), \hspace{1cm} 0\leq t<\infty
\end{equation}
with some constant $0<\mu, C<\infty$ independent of $X_0$. Furthermore, for any $\frac{3}{4}<\eta<1$, 
\begin{equation}    \label{Ton38}
\|A^\eta X(t)\| \leq [1+(\iota_0+\iota_\eta)\|X_0\|]  t^{-\eta}, \hspace{1cm} 0<t<\infty,
\end{equation}
and
\begin{equation}  \label{Ton39}
\|A X(t)\| \leq \chi(\|X_0\|)  (1+ t^{-1}), \hspace{1cm} 0<t<  \infty, 
\end{equation}
where $\chi$ is some real-valued positive continuous function on $[0,\infty)$.  
\end{theorem}

\begin{proof}
For $0\leq r<\infty$, we define a ball in $E$:
$$B_r=\{ \xi \in V; \|\xi\|\leq r\}.$$
 and a function $\tau_1$ on $[0,\infty)$:
$$\tau_1(r)=\min_{x\in B_r} \tau (x).$$
Since $\tau$ is a positive continuous function on $[0,\infty)$,  the function $\tau_1$ is also positive.

Since  $X$ is the local strong solution of \eqref{Ton11} on $[0, \tau(\|X_0\|)],$  we have 
$$\tau(\|X_0\|) \geq \tau_1(C(\|X_0\|+1))$$
due to   \eqref{Ton33}.

Let us consider  the Cauchy problem  \eqref{THM6} with
$$t_0=\tau(\|X_0\|)-\frac{\tau_1(C(\|X_0\|+1))}{2}, \quad \zeta=X(t_0).$$ 
Theorem  \ref{THM6} then provides that  \eqref{THM6} possesses a local solution, say  $\bar X,$ on an interval $[t_0, t_0+\tau(\|X(t_0)\|)]$. 
 Furthermore, by Theorem \ref{THM8}, $\bar X$ satisfies \eqref{Ton33} on that interval.

Since $ \zeta=X(t_0)$, it belongs to $B_{C(\|X_0\|+1)}$ due to   \eqref{Ton33}. 
Hence, by the definition of the function $\tau_1$, 
$$[t_0, t_0+\tau_1(C(\|X_0\|+1))] \subset [t_0, t_0+\tau(\|X(t_0)\|)].$$
This shows that $\bar X$ is well-defined on $[t_0, t_0+\tau_1(C(\|X_0\|+1))]$. 
Note that $X$ is defined on $[0, \tau(\|X_0\|)]$, and therefore on $[t_0, t_0+\frac{1}{2}\tau_1(C(\|X_0\|+1))].$
 The uniqueness of solutions then implies that
$$X(t)=\bar X(t), \hspace{1cm} t_0\leq t\leq  \tau(\|X_0\|)=t_0+\frac{1}{2}\tau_1(C(\|X_0\|+1)).$$
This means that we have constructed a local solution, still denoted by $X$,  to \eqref{Ton11} on 
$[0, t_0+\tau_1(C(\|X_0\|+1))]=[0, \tau(\|X_0\|)+\frac{1}{2}\tau_1(C(\|X_0\|+1))].$

Thanks to the {\em priori} estimate  \eqref{Ton33}, this procedure can be continued infinitely. Each time the local solution is extended over the fixed length $\frac{1}{2}\tau_1(C(\|X_0\|+1))$ of interval. Thus, the solution $X$ is prolonged to $[0,\infty)$. 
 In addition, the estimates \eqref{Ton37},   \eqref{Ton38}, and  \eqref{Ton39}    follow  from \eqref{Ton32}, \eqref{Ton19},  and  \eqref{Ton20}, respectively.
The proof is  complete. 
\end{proof}

\section{Regular dependence on initial data}  \label{initial data}
In this section, we show that solutions of \eqref{Ton1} are continuously dependent on initial values.

On account of Theorem \ref{THM8}, for any initial value  $X_0=\left(
\begin{matrix} 
u_0\\
v_0
\end{matrix}\right) \in V,$ 
where $V$ is defined by \eqref{Ton7},  the system  \eqref{Ton1} coupled with the Neumann condition \eqref{Ton2} possesses a unique global nonnegative solution  $X(\cdot,X_0)=\left(
\begin{matrix} 
u\\
v
\end{matrix}\right) \in E.$ 
The following theorem shows that when $X_0$ is ``close'' to $Y_0$, so is $X(t,X_0)$ to $X(t,Y_0)$ at time $t$.

\begin{theorem}  \label{THM10}
For any $\frac{3}{4}<\eta<1$, there exists $0<c_\eta<\infty$ such that 
\begin{equation*}  
\|X(t,X_0)-X(t,Y_0)\| \leq \iota_0 \|X_0-Y_0\| e^{ c_\eta[t+\frac{\|X_0\|+\|Y_0\|}{1-\eta}  t^{1-\eta}]}
\end{equation*}
for all $0\leq t<\infty$ and $ X_0, Y_0\in V.$ 
\end{theorem}

\begin{proof}
Let $ X_0, Y_0\in V$. Since
\begin{equation*}
X(t,\cdot)=e^{-tA} (\cdot) +\int_0^t e^{-(t-s)A}\bar F(X(s,\cdot))ds, \hspace{1cm} 0\leq t<\infty,
\end{equation*}
we have
\begin{align*}
& \|X(t,X_0)-X(t,Y_0)\|  \\
\leq & \|e^{-tA}\| \|X_0-Y_0\| +\int_0^t \|e^{-(t-s)A}\|  \| \bar F(X(s,X_0))-\bar F(X(s,Y_0))\|ds.
\end{align*}
Thank to \eqref{Ton5}, Lemma \ref{THM3},  and Theorem \ref{THM8}, it follows that 
\begin{align*}
& \|X(t,X_0)-X(t,Y_0)\|  \\
\leq & \iota_0 \|X_0-Y_0\| + \int_0^t  c_\infty[1+\|X(s,X_0)\|_{\infty}+\|X(s,Y_0)\|_{\infty}]  \\
& \hspace{2.5cm} \times   \| X(s,X_0)-X(s,Y_0)\|ds, \hspace{1cm} 0\leq t<\infty.
\end{align*}
In view of \eqref{Ton23} and  \eqref{Ton38}, for any $\frac{3}{4}<\eta<1$, there exist $0<c_\eta, c_{1\eta}<\infty$ such that
\begin{align*}
& \|X(t,X_0)-X(t,Y_0)\|  \\
\leq & \iota_0 \|X_0-Y_0\| + \int_0^t  c_{1\eta}[1+\|A^\eta X(s,X_0)\|+\|A^\eta X(s,Y_0)\|]  \\
& \hspace{2.5cm} \times   \| X(s,X_0)-X(s,Y_0)\|ds\\
\leq & \iota_0 \|X_0-Y_0\| + \int_0^t  c_\eta[1+(\|X_0\|+\|Y_0\|) s^{-\eta}]  \\
& \hspace{2.5cm} \times   \| X(s,X_0)-X(s,Y_0)\|ds,  \hspace{1cm} 0\leq t<\infty.
\end{align*}
The  Gronwall inequality then provides that
\begin{align*}
\|X(t,X_0)-X(t,Y_0)\|  
\leq & \iota_0 \|X_0-Y_0\| e^{\int_0^t c_\eta[1+(\|X_0\|+\|Y_0\|) s^{-\eta}]ds}  \\
=  & \iota_0 \|X_0-Y_0\| e^{ c_\eta[t+\frac{\|X_0\|+\|Y_0\|}{1-\eta}  t^{1-\eta}]},  \hspace{1cm} 0\leq t<\infty.
\end{align*}
We thus complete the proof.
\end{proof}

\section{Dynamical system}  \label{Dynamical system}
In this section, we construct a dynamical system for the coat model \eqref{Ton1}. Furthermore, we show that the dynamical system enjoys an exponential attractor having finite fractal dimension.

Let $X(t,X_0)$ be the solution of \eqref{Ton11}.
By setting 
$$S(t)X_0=X(t,X_0), \hspace{1cm} X_0\in V,$$
we define a nonlinear   semigroup $S$ acting on $V$. By the continuity  of  solutions in time as well as  Theorem \ref{THM10}, the semigroup is continuous from $[0,\infty)\times V$ to $V$.  Thus, the equation  \eqref{Ton11}  determines a continuous dynamical system $(S,V,E)$.

For the construction of an exponential attractor, we start with the following 
 proposition.

\begin{proposition}  \label{THM11}
There exists a constant $0<\varrho<\infty$ such that for all bounded set $B$ of $V,$ there is a time $t_B$ depending on $B$ such that
\begin{equation*}  \label{}
\sup_{X_0\in B} \sup_{t\geq t_B}  \|A S(t)X_0\| \leq \varrho.
\end{equation*}
\end{proposition}

\begin{proof}
Let $B$ be a bounded set  of $V$.  By  \eqref{Ton37} of Theorem \ref{THM9}, there exist $\varrho_1$ independent of $B$ and a time $t_B^*$ depending on $B$ such that
\begin{equation}  \label{Ton40}
\sup_{X_0\in B}\sup_{t\geq t_B^*}\|S(t)X_0\| \leq \varrho_1.
\end{equation}

Let $t_B^*\leq t_0<\infty$. Consider the problem \eqref{Ton18} with an initial value $\zeta=S(t_0)X_0$, $X_0\in B$. By \eqref{Ton20} of Theorem \ref{THM6}, we have for every $t_0<t\leq  t_0+\tau(\|S(t_0)X_0\|),$
\begin{equation}  \label{Ton41}
\|A S(t)S(t_0)X_0\| \leq \chi(\|S(t_0)X_0\|)  [1+ (t-t_0)^{-1}].
\end{equation}

Put
$$\kappa_1=\min_{0\leq x\leq \varrho_1} \tau(x) \quad \text{ 
and  }  \quad  \kappa_2=\max_{0\leq x\leq \varrho_1} \chi(x)\in (0,\infty).$$
   Then, $ \kappa_1$ and $\kappa_2$ are positive  because  $\tau$ and $\chi$ are   continuous positive functions on $[0,\infty)$. In addition, 
 these constants  are independent of $B$.
We therefore observe from \eqref{Ton40} and  \eqref{Ton41}  that  for all $ t_B^*\leq t_0<\infty $ and $ t_0<t\leq  t_0+\kappa_1$, 
\begin{equation*}  
\|A S(t)S(t_0)X_0\| \leq \kappa_2  [1+ (t-t_0)^{-1}].
\end{equation*}

We  apply this with $t_0=t-\kappa_1$. Then, for all $t_B^*+ \kappa_1\leq t<\infty$, 
$$\|A S(t)S(t-\kappa_1)X_0\| \leq \kappa_2  (1+ \kappa_1^{-1}).$$
Since $S(t)S(t-\kappa_1)X_0=S(2t-\kappa_1)X_0$, this inequality means that
$$\|A S(t)X_0\| \leq \kappa_2  (1+ \kappa_1^{-1}), \hspace{1cm} 2t_B^*+\kappa_1\leq t<\infty.$$
Thus,
$$\sup_{X_0\in B} \sup_{t\geq 2t_B^*+\kappa_1}  \|A S(t)X_0\| \leq \kappa_2  (1+ \kappa_1^{-1}).$$
The proof is therefore complete.
\end{proof}

From Proposition \ref{THM11}, we can show existence of   an absorbing  set for the dynamical system $(S,V,E)$.

\begin{theorem}  \label{THM12}
The  closed ball $\mathcal B$ in $\mathcal D(A)$ defined by 
\begin{equation*}
\mathcal B=\{x\in \mathcal D(A); \|Ax\| \leq  \varrho\},
\end{equation*}
where $\varrho$ is the constant in Proposition \ref{THM11}, is a compact and absorbing set for $(S,V,E)$.
\end{theorem}

\begin{proof}
The proof is obvious due to some well-known results.  Indeed, the absorbing property of  $\mathcal B$ follows from Proposition \ref{THM11}. Furthermore, since $\mathcal D(A)=H^2_N(\Omega) \times H^2_N(\Omega )$ (see \eqref{Ton8}), $\mathcal D(A)$  is compactly embedded in $E$ (see \cite[Theorem 4.10.1]{triebel}).
Thus, $\mathcal B$  is a compact set of $E$ (see \cite[Proposition 6.4]{yagi}).
\end{proof}

From the absorbing set in Theorem \ref{THM12}, we can construct an absorbing, compact, and invariant set for $(S, V, E).$ Indeed, since the  closed ball $\mathcal B$ is an absorbing set,   there exists a time $t_{\mathcal B}$ such that 
$$S(t)\mathcal B \subset \mathcal B, \hspace{1cm} t_{\mathcal B}\leq t<\infty.$$
We then put
  \begin{equation}  \label{Ton42} 
\mathcal V=\overline{\cup_{t_{\mathcal B}\leq t<\infty} S(t) \mathcal B} \subset \mathcal B \hspace{1cm} (\text{closure in } E).
\end{equation}

\begin{theorem} 
The set $\mathcal V$ defined by \eqref{Ton42} is an absorbing, compact, and invariant set of $(S, V, E).$
\end{theorem}

\begin{proof}
Since $\mathcal B$  is an absorbing and compact set of $E$, so is $\mathcal V$. In addition, we have
\begin{align*}
S(t) \mathcal V=&S(t) \overline{\cup_{t_{\mathcal B}\leq r<\infty} S(r) \mathcal B}   \\
\subset & \overline{\cup_{t_{\mathcal B}\leq r<\infty} S(t)S(r) \mathcal B} \subset \mathcal V, \hspace{1cm} 0\leq t<\infty.
\end{align*}
This means that the set $\mathcal V$ is    invariant. 
\end{proof}

In this way, we observe that the behavior of the dynamical system $(S,V,E)$ is reduced to that of a dynamical system $(S, \mathcal V, E)$.
Let us now show that the dynamical $(S, \mathcal V, E)$ enjoys an exponential attractor having  finite fractal dimension. For this purpose, we want to use Theorem \ref{THM2}. 

First, let us show the squeezing property for $S(t)$ (see the paragraph before Theorem \ref{THM2} for the definition of squeezing property).

\begin{proposition}   \label{THM13}
For every  $0<t_0<\infty,$  the operator  $S(t_0)$  has the squeezing property with some $0<\delta<\frac{1}{4}$  and an orthogonal operator $P$ of finite rank $N$. 
\end{proposition}

\begin{proof}
The arguments  are quite similar to one in \cite{yagi}.  
Since $A$ is a positive definte self-adjoint operator of $E$ and
$$\mathcal D(A)=
H^{2}_N(\Omega ) \times H^{2}_N(\Omega ) \hspace{1cm} (\text{see } \eqref{Ton8}),$$
the operator $A$ has eigenvalues $\lambda_n $ and corresponding eigenvectors $e_n (n=1, 2, 3\dots)$ such that
\begin{itemize}
  \item the sequence $\{\lambda_n\}_{n=1}^\infty $ is increasing and tends to infinity as $n$ tends to infinity 
  \item the sequence  $\{e_n\}_{n=1}^\infty $ is an orthogonal basis of $E$
\end{itemize}

 Consider an $N$-dimension subspace of $E$ 
$$E_N=\text{Span} \{e_1,e_2,e_3\dots e_N\}$$ 
with some integer $N,$ and the orthogonal projection $P\colon E\to E_N$.  Let's fix $0<t_0<\infty $. To prove that $S(t_0)$  has the squeezing property, it suffices to show  existence of a constant $0<\delta<\frac{1}{4}$  such that
if 
\begin{equation}    \label{Ton43} 
\| P(S(t_0)(x)-S(t_0)(y))\|< (1-P)(S(t_0)(x)-S(t_0)(y)) 
\end{equation}
for some $x, y \in \mathcal V,$
then 
\begin{equation}    \label{Ton44}
\|S(t_0)(x)-S(t_0)(y)\| \leq \delta \|x-y\|.
\end{equation}

Indeed, since $\{e_n\}_{n=1}^\infty $ is an orthogonal basis of $E$,
$$S(t_0)(x)-S(t_0)(y)= \sum_{n=1}^\infty \alpha_i e_i,$$
where $\alpha_i=\langle S(t_0)(x)-S(t_0)(y), e_i \rangle$ ($\langle \cdot,\cdot \rangle $ is the scalar product in $E$).  Then, \eqref{Ton43} gives
$$\sum_{n=1}^N  \alpha_i^2 <   \sum_{n=N+1}^\infty  \alpha_i^2.$$
Therefore, 
\begin{align}
\|S(t_0)(x)-S(t_0)(y)\|^2   
&< 2  \sum_{n=N+1}^\infty  \alpha_i^2 \notag\\
&= 2 \|(1-P)(S(t_0)(x)-S(t_0)(y))\|^2.   \label{Ton45}
\end{align}

In the meantime, using the solution formula in Definition \ref{THM4}, it is easily seen that
\begin{align*}
(1-P)&(S(t_0)(x)-S(t_0)(y))=  e^{-t_0A} (1-P)(x-y) \\
&+\int_0^{t_0} e^{-(t_0-s)A} (1-P) (  \bar F(S(s)(x))-\bar F(S(s)(y))   )ds.
\end{align*}

The norm of the first term in the right-hand side of the latter equality can be estimated as 
\begin{align*}
\|e^{-t_0A} (1-P)(x-y)\| 
&=\|e^{-t_0A} (1-P)(  \sum_{n=1}^\infty \langle x-y, e_n \rangle e_n )\| \\ 
&=\|e^{-t_0A}   \sum_{n=N+1}^\infty \langle x-y, e_n \rangle e_n \| \\
&=\|   \sum_{n=N+1}^\infty \langle x-y, e_n \rangle e^{-t_0\lambda_n} e_n \| \\
&\leq   e^{-t_0\lambda_{N+1}} \|   \sum_{n=N+1}^\infty \langle x-y, e_n \rangle  e_n \| \\
&\leq   e^{-t_0\lambda_{N+1}} \|  x-y \|.
\end{align*}

Similarly, we have an estimate for the norm of the second term:
\begin{align*}
& \| \int_0^{t_0} e^{-(t_0-s)A} (1-P) (  \bar F(S(s)(x))-\bar F(S(s)(y))   )ds \|   \\
& \leq   \int_0^{t_0} \| e^{-(t_0-s)A} (1-P) (  \bar F(S(s)(x))-\bar F(S(s)(y))   )\|ds  \\
&  \leq  \int_0^{t_0} e^{-(t_0-s)\lambda_{N+1}} \| \bar F(S(s)(x))-\bar F(S(s)(y)) \|ds.  
 \end{align*}
 Lemma \ref{THM3} and the inequalities  \eqref{Ton23} and \eqref{Ton38} then give  
\begin{align*}
& \| \int_0^{t_0} e^{-(t_0-s)A} (1-P) (  \bar F(S(s)(x))-\bar F(S(s)(y))   )ds \|   \\
&  \leq c_\infty C \int_0^{t_0} e^{-(t_0-s)\lambda_{N+1}} 
[1+\|A^\eta S(s)(x)\|+\|A^\eta S(s)(y)\|]  \\
& \hspace{2cm} \times \| S(s)(x)-S(s)(y) \|ds\\
&  \leq c_\infty C \int_0^{t_0} e^{-(t_0-s)\lambda_{N+1}} 
[1+   \{1+(\iota_0+\iota_\eta)\|x\|\}  s^{-\eta}    
     +\{1+(\iota_0+\iota_\eta)\|y\|\}  s^{-\eta}] \\
& \hspace{2cm} \times
\| S(s)(x)-S(s)(y) \|ds.    
 \end{align*}
By using Theorem \ref{THM10} and the fact that $x$ and $ y$ belong to the bounded set $ \mathcal V,$ there exists a constant $C_1>0$ such that
\begin{align*}
& \| \int_0^{t_0} e^{-(t_0-s)A} (1-P) (  \bar F(S(s)(x))-\bar F(S(s)(y))   )ds \|   \\
&  \leq C_1 \int_0^{t_0} e^{-(t_0-s)\lambda_{N+1}} 
(1+ s^{-\eta})ds   \| x-y\|   \\
&=   C_1 \left\{\frac{1-e^{-t_0\lambda_{N+1}}}  {\lambda_{N+1}} +\int_0^{t_0} e^{-(t_0-s)\lambda_{N+1}} 
 s^{-\eta}ds     \right\} \| x-y\|  
 \end{align*}

The above estimates for the two terms imply that 
\begin{align*}
& \|(1-P)(S(t_0)(x)-S(t_0)(y))\|  \\
& \leq \left [ e^{-t_0\lambda_{N+1}}  +   C_1 \left\{\frac{1-e^{-t_0\lambda_{N+1}}}  {\lambda_{N+1}} +\int_0^{t_0} e^{-(t_0-s)\lambda_{N+1}} 
 s^{-\eta}ds     \right\}
 \right]      \| x-y\|.  \notag
\end{align*}
Because of $\lim_{N\to \infty} \lambda_{N+1} =\infty$, it is easily seen that
$$\lim_{N\to \infty}  \int_0^{t_0} e^{-(t_0-s)\lambda_{N+1}}  s^{-\eta}ds=0. $$
Therefore, 
\begin{align}
\|(1-P)(S(t_0)(x)-S(t_0)(y))\|  
 \leq \frac{1}{6}\| x-y\|  \label{Ton46}
\end{align}
if $N$ is sufficiently large.
Thus,  \eqref{Ton44} follows from   \eqref{Ton45} and  \eqref{Ton46}.
The proposition thus has been proved.
\end{proof}

Second, let us show that $S(\cdot)$ is Lipschitz continuous in the sense of the following proposition.

\begin{proposition}   \label{THM14}
For every  $0<t_0<\infty$, there exists $L>0$ such that
\begin{align*}
\|S(t)x-S(s)y\| & \leq L[\|x-y\| + (t-s)], \hspace{1cm} 0\leq s\leq t\leq t_0,  x, y\in \mathcal V.
\end{align*}
\end{proposition}

\begin{proof}

Since $\mathcal V$ is bounded in $\mathcal D(A)$ and therefore in $E$, 
Theorem \ref{THM10} provides that 
\begin{equation*} 
\|S(t)x-S(t)y\| \leq C_{1,\mathcal V} \|x-y\|, \hspace{1cm} 0\leq t\leq t_0, x, y\in \mathcal V,
\end{equation*}
with some constant $C_{1,\mathcal V}>0.$

In the meantime, we have
$$\|S(t)y-S(s)y\|=\|\int_s^t [\bar F(S(u)y)-AS(u)y]du\|, \hspace{1cm} 0\leq s<t\leq t_0,  y\in \mathcal V.$$
Since $\mathcal V$ is invariant with respect to $S$ and bounded in $\mathcal D(A)$, it is easily seen from the latter equality that 
$$\|S(t)y-S(s)y\|   \leq C_{2,\mathcal V} (t-s), \hspace{1cm} 0\leq s<t\leq t_0,  y\in \mathcal V$$
with some constant $C_{2,\mathcal V}>0.$ Thus, we observe that
\begin{align*}
\|S(t)x-S(s)y\| & \leq \|S(t)x-S(t)y\| + \|S(t)y-S(s)y\| \\
& \leq C_{1,\mathcal V} \|x-y\| + C_{2,\mathcal V} (t-s), \hspace{1cm} 0\leq s<t\leq t_0,  x, y\in \mathcal V.
\end{align*}
The proof is thus complete.
\end{proof}

We are now ready to state  results on exponential attractors.

\begin{theorem}   \label{THM15}
Let $0<t_0<\infty$ and $\delta, N, $ and $L$ be constants defined in Propositions \ref{THM13} and \ref{THM14}.
Then, for any $0<\theta<1-2\delta$, there exists an exponential attractor $\mathcal A_\theta$ for the dynamical system $(S, \mathcal V, E)$   such that
$$h(S(t)\mathcal V,\mathcal A_\theta) \leq \frac{ 
\sup_{x\in \mathcal V} \|x\|}{2\delta+\theta 
} 
e^{\frac{t\log (2\delta+\theta) }{t_0}}, \hspace{ 1cm} 0<t <\infty.$$

Furthermore,  $\mathcal A_\theta$ has an finite fractal dimension  estimated by
$$d_F(A_\theta) \leq 1 + N \max\{  -\frac{  \log (\frac{3L}{\theta}+1)    }  {\log (2\delta+\theta)}, 1   \}.$$
\end{theorem}

\begin{proof}
Propositions \ref{THM13} and \ref{THM14} show that all the assumptions of Theorem \ref{THM2} take place. Thus, the conclusions in Theorem \ref{THM15} follow from one in Theorem \ref{THM2}. 
\end{proof}

\section{An example}  \label{example}

Let us consider an example of the system \eqref{Ton1}. For numerical simulations, we use the finite difference schemes presented in \cite{Garvie}.

Set
$\gamma=15$, $\beta=1.5$, $\rho=13$, $\alpha=7$, $a=103$, $b=77$, $k=0.125.$ Consider \eqref{Ton1} in the two-dimensional space with random initial value 
$(u_0,v_0)$  near the stationary solution. (In fact, $23<u_0<24$ and $24<v_0<25$.)

We calculate $26\times 26$ values of $(u(T),v(T))$ in the rectangular $[0,25]\times [0,25],$ where $T=90$ and $T=150$.  Since the geometry of the concentration of the activator $u$ can be interpreted as describing the coat pattern of a specific animal, we  illustrate this function in Figure \ref{Fig1}.

 \begin{figure}[H]
 \begin{center}
\includegraphics[scale=0.4]{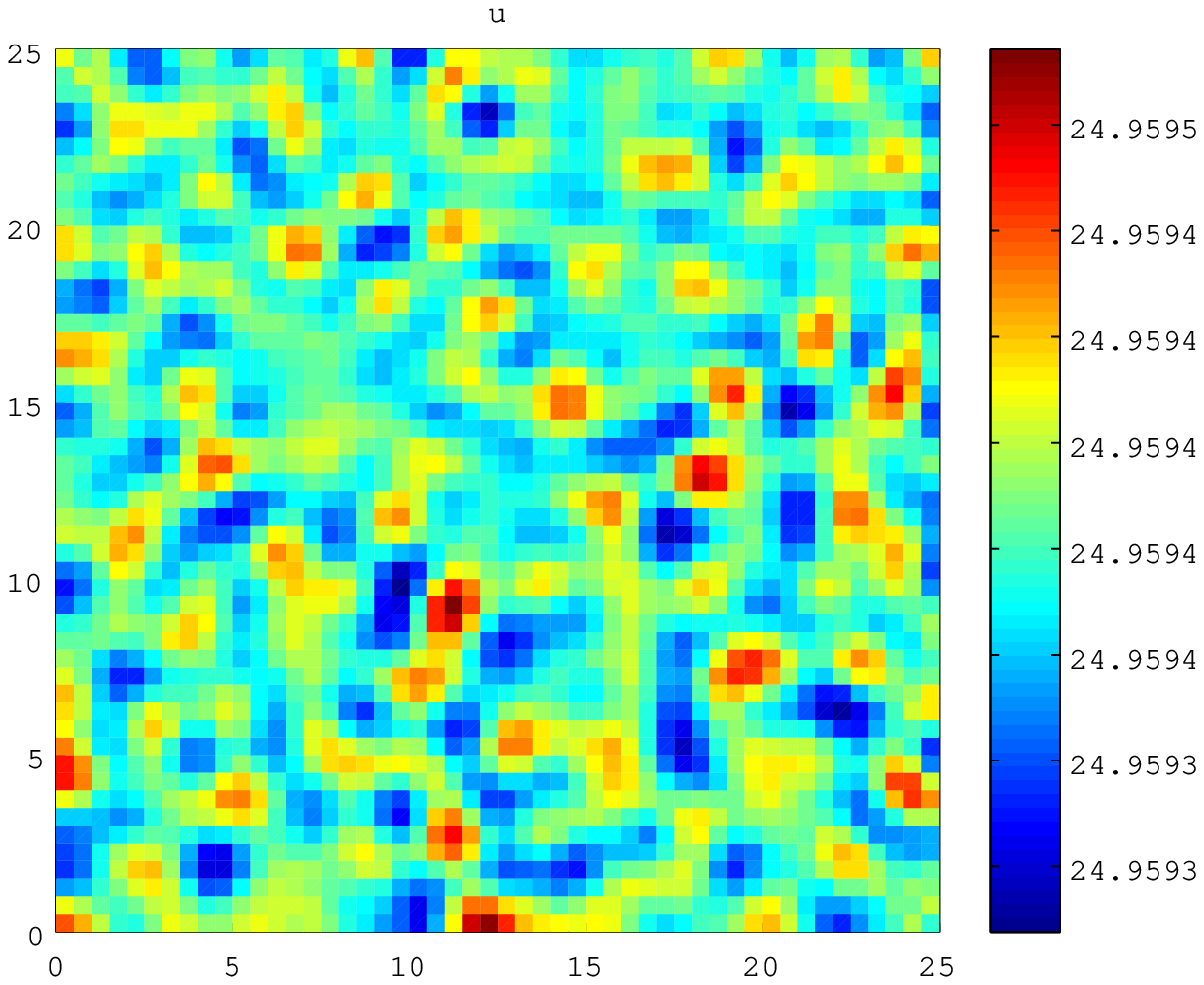}
\includegraphics[scale=0.4]{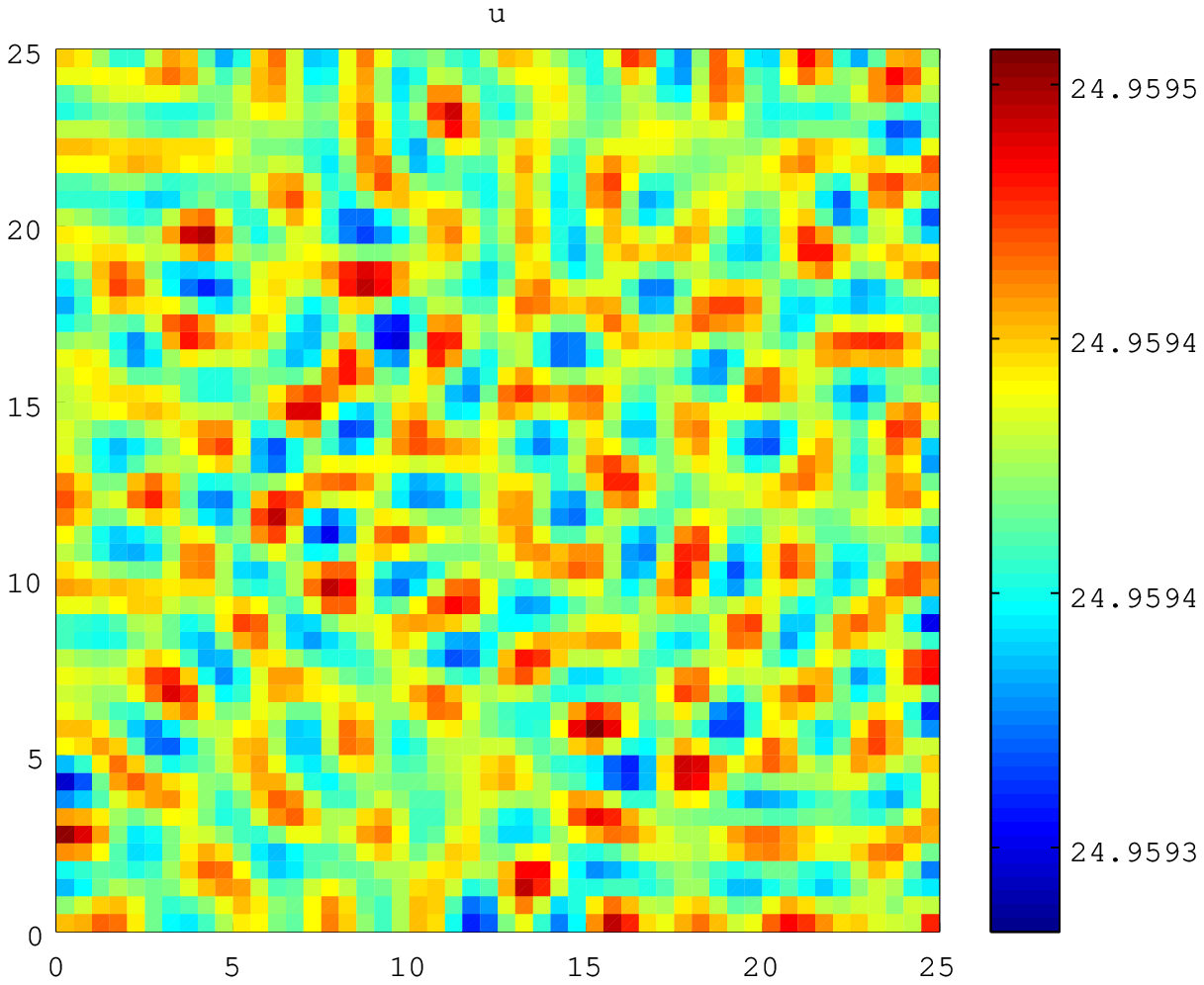}
 \caption{Simulation of the system \ref{Ton1} for the two-dimensional case $\Omega=[0,25]\times [0,25]$ with  $\gamma=15$, $\beta=1.5$, $\rho=13$, $\alpha=7$, $a=103$, $b=77$, $k=0.125.$ The two figures show the $u$-component at $T=90$ and $T=150$ of a solution starting  at a random perturbation near the stationary solution, where color represents the concentration of the activator.} 
 \label{Fig1}
 \end{center}
 \end{figure}

\end{document}